\documentclass{amsart}
\usepackage{amsmath,amssymb,amsthm}
\usepackage{enumerate}
\usepackage{cite}
\usepackage[hidelinks]{hyperref}

\DeclareMathOperator{\sym}{sym}

\def \N {\mathbb{N}}

\def \Y {\mathbb{Y}}

\def \dropline {\vspace{\baselineskip}}

\def \calA {\mathcal{A}}
\def \calB {\mathcal{B}}
\def \calC {\mathcal{C}}
\def \calD {\mathcal{D}}
\def \calL {\mathcal{L}}

\def \calN {\mathcal{N}}
\def \calW {\mathcal{W}}

\def \fC {\mathfrak{C}}

\def \dos {d_{0s}}

\def \ds {d_{s}}

\def \leqs {\leqslant_s}
\def \esi {\emptyset\leqs}

\def \mult {\mathfrak{m}}

\newcommand{\fullenum}[1]{\mathcal{U}_{#1}}

\newcommand{\dosover}[2]{\dos (#1/#2)}
\newcommand{\otn}[2]{#1_{1},...,#1_{#2}}
\newcommand{\set}[1]{\{#1\}}
\newcommand{\setarg}[2]{{\{#1\ |\ #2\}}}

\newcommand{\clq}[1]{(#1,wit(#1))}
\newcommand{\clqwit}[2]{(#1,\set{#2})}

\newtheorem{lemma}{Lemma}[section]
\newtheorem{theorem}[lemma]{Theorem}
\newtheorem*{theorem*}{Theorem}
\newtheorem{prop}[lemma]{Proposition}
\newtheorem*{prop*}{Proposition}
\newtheorem{Cor}[lemma]{Corollary}
\newtheorem*{Cor*}{Corollary}
\newtheorem*{question*}{Question}
\theoremstyle{definition}
\newtheorem{obs}[lemma]{Observation}
\newtheorem*{obs*}{Observation}
\newtheorem{definition}[lemma]{Definition}
\newtheorem{remark}[lemma]{Remark}
\newtheorem*{notation}{Notation}
\newtheorem*{remark*}{Remark}
\newtheorem{fact}[lemma]{Fact}

\newcommand{\itemlabel}[1]{\renewcommand{\labelitemi}{#1}}

\usepackage[all]{xy}

\newcommand{\outsource}[3]{{#1}_{#2}^{#3}}

\newcommand{\witclq}[2]{\set{#1}^{#2}}
\def \CC {\mathcal{C}}
\def \CCS {\mathcal{C}_s}
\def \strong {\leqslant}
\def \isoext {\rightsquigarrow}

\newcommand{\wtclq}[2]{#1^{(#2)}}

\DeclareMathOperator{\acl}{acl}

\title{On reducts of Hrushovski's construction - the non-collapsed case}
\author{Assaf Hasson}
\thanks{The first author was partially supported by an Israel Science
  Foundation grant number 1156/10.}
\address{Dept. of Math.\\ Ben Gurion University\\ POB 653\\ Beer Sheva, 84105, Israel.}
\email{hassonas@math.bgu.ac.il}
\author{Omer Mermelstein}
\address{Dept. of Math.\\ Ben Gurion University\\ POB 653\\ Beer Sheva, 84105, Israel.}
\email{omermerm@math.bgu.ac.il}

\begin{document} 

\begin{abstract}
 We show that the rank $\omega$ structure obtained by the non-collapsed version of Hrushovski's amalgamation construction has a proper reduct. We show that this reduct is the Fra\"iss\'e-Hrushovski limit of its own age with respect to a pre-dimension function generalising Hrushovski's pre-dimension function. It follows that this reduct has a unique regular type of rank $\omega$, and we prove that its geometry is isomorphic to the geometry of the generic type in the original structure. We ask whether our reduct is bi-interpretable with the original structure and whether it, too, has proper reducts with the same geometry. 
\end{abstract}

\maketitle

\section{Introduction}

\subsection{Background}
In the late 70s of the 20$^{\text{th}}$ century Zilber formulated the conjecture that the $\acl$-geometries associated with strongly minimal sets come in three flavours: trivial (that of an infinite set with no structure), linear (that of an infinite projective space over a division ring) or field-flavoured (that of an algebraic curve over an algebraically closed field). Though refuted by Hrushovski in the late 80s, \cite{Hns} and \cite{HrushovskiSecond}, the theme of classifying structures (and, more generally, types) according to their associated geometries remains a central idea and an important tool in model theory -- reaching way  beyond the realm of stability.  

In \cite{Hns} Hrushovski introduces a technique for constructing non-locally modular strongly minimal sets not interpreting a group (let alone a field). However, the geometries of all the new strongly minimal sets constructed in \cite{Hns} shared a common combinatorial property, known as CM-triviality,  the mildest -- in a precise technical sense -- possible relaxation of modularity (\cite{NotePillayAmple}). Thus, while  refuting Zilber's original idea that the classification of strongly minimal geometries could proceed by classifying the algebraic structures they interpret, Hrushovski's work did not destroy altogether the hope that some classification of strongly minimal geometries may still exits. A more serious hurdle on the way to such a classification was Hrushovski's work in \cite{HrushovskiSecond}, where a modification of the original construction from \cite{Hns} allowed, given any collection $\{T_i\}_{i=1}^\infty$ of strongly minimal theories satisfying a minor technical requirement (DMP),  to 
construct a strongly 
minimal theory $T$ having each of the theories $T_i$ as a reduct (allowing, for example the construction of a strongly minimal set with one reduct an algebraically closed field of characteristic $2$, and one reduct an algebraically closed field of characteristic $3$). 

Hrushovski's fusion construction of \cite{HrushovskiSecond} seems to undermine the idea of classifying the geometries of strongly minimal sets along the simple line suggested by Zilber's conjecture, i.e., it proves that there cannot be a fixed collection of ``classical'' (whichever way this term is interpreted) structures such that the geometry of any strongly minimal set is isomorphic to the geometry of one of those structures. Nevertheless, in his paper, Hrushovski suggests, informally, that the geometry of the fusion theory $T$  is flat \emph{over} the data (i.e., over the collection of theories, $T_i$, taking place in the fusion). This leaves some hope that the geometry of the fusion could be structurally understood in terms of its building blocks, namely the geometries of the different theories, $T_i$. 

This, naturally, leads to the question: if an attempt is made to formulate a conjectural structural theory for the geometry of strongly minimal sets, in what terms should it be formulated. Are there (in some sense yet to be formulated) \emph{smallest} (or \emph{prime}) strongly minimal theories which could serve as basic building blocks of any strongly minimal theory? Observing that the relation of interpretability is a quasi-order on first order theories, it seems reasonable to think of a strongly minimal theory, $T$, as \emph{geometrically-prime} if it is minimal (with respect to the quasi-order of interpretability) within the class of strongly minimal theories with the same geometry as $T$, that is, if $T$ is interpretable in any strongly minimal theory $T'$ interpretable in $T$, whose associated geometry is isomorphic to that of $T$. 

In that sense the theory of pure equality is certainly prime. By Hrushovski's structure theory for locally modular regular types, \cite{Hr4}, so are the theories  of infinite vector (or affine) spaces over division rings. Since any non-trivial locally modular strongly minimal theory interprets a vector space, it follows that -- up to bi-interpretability -- these are the only prime locally modular theories. By Hrushovski's field configuration theorem, a strongly minimal theory whose $\acl$-geometry is isomorphic to that of an algebraically closed field, $F$, interprets an algebraically closed field, and since an  algebraically closed field is interpretable in its own combinatorial geometry by \cite{EvHr} this field must be isomorphic to $F$. Therefore algebraically closed fields are also geometrically prime. Zilber's trichotomy conjecture would have implied that any non-locally-modular theory interpretable in ACF is bi-interpretable with ACF, whence it would follow that ACF are the only geometrically prime 
non-locally modular strongly minimal theories. 

In view of the above, we find it natural to ask: are there more geometrically prime strongly minimal sets? Since the conjecture that any non-locally modular theory interpretable in ACF interprets an algebraically closed field is still open (proved by Rabinovich, \cite{Rabinovich}, for reducts of ACF). It seems, therefore, that the only known places to be looking for new geometrically prime strongly minimal theories is among Hrushovski's constructions. The present work is a first step in that investigation. 

We do not know whether Hrushovski's new strongly minimal sets are geometrically prime. For the sake of simplicity, our work focuses on the, so called, \emph{ab initio} construction, with the hope that a good understanding of that construction could help in a similar analysis of the fusion construction. Hrushovski's original construction is implemented in a relational language with a unique ternary relation, and it seems reasonable to start our investigation there. As will be clear from the body of the work, everything we do in this paper could work in, essentially, the same way for the construction with a unique $n$-ary relation (for $n\ge 3$). Indeed, in some respects the construction for $n\ge 5$ is simpler. On the other hand, the results of \cite{DavidMarcoOne} and \cite{DavidMarcoTwo} suggest that the structures obtained by a construction with more than one relation are not geometrically prime. 

Finally, though the present work is only a first step, the methods developed in this paper suggest that a classification of all reducts of Hrushovski's new strongly minimal set could be a feasible project. Since the structure in the $3$-ary relation language is interpretable in any structure obtained via a similar construction in a language containing at least one relation of arity at least $3$, it seems reasonable that this should be the test case for our investigation. 

In the present work we deal solely with the non-collapsed version of Hrushovski's construction. A second part of this work, dealing with the strongly minimal version of the construction will appear separately. 

\subsection{Results}
In order to adapt the above discussion to the context of Hrushovski's non-collapsed structure we note that for an ($\omega$-stable, for the sake of simplicity) structure $\mathcal{M}$ with a unique \emph{generic} regular $p\in S_1(\emptyset)$, the $p$-geometry $\rm{cl}_p$ defines a geometry on $\mathcal{M}$, which we call (\emph{ad hoc} but unambiguously) the geometry of $\mathcal{M}$. The natural generalisation of the notion of geometrically prime seems, therefore, to be: $\mathcal{M}$ is geometrically prime if whenever $\mathcal{M}'$ is a rank-preserving reduct of $\mathcal{M}$ such that the geometry of $\mathcal{M}'$ is isomorphic to the geometry of $\mathcal{M}$, then $\mathcal{M}$ is interpretable in $\mathcal{M}'$. Obviously, if $\mathcal{M}$ admits no rank and geometry preserving reducts then $\mathcal{M}$ is automatically geometrically prime (e.g., the theory of pure equality and the theory of a vector space over a prime field are geometrically prime, but ACF is not). It is thus a natural first step to ask:  

\begin{question*}
Does the non-collapsed version of Hrushovski's construction (from \cite{Hns}) admit a proper reduct with an isomorphic geometry\footnote{This question was also asked by D. Evans in private communication.}?
\end{question*}

% \begin{remark*}From here on, everything is done in the non-collapsed case, but there should not be a significant problem with the collapse. We work in Hrushovski's construction for a ternary relation $R$ such that $(a,b,c)\in R$ implies that the elements $a,b,c$ are distinct.
% \end{remark*}

We answer this question in the positive. In section \ref{symmetric} we show that the symmetric reduct of Hrushovski's structure is a proper reduct and that its geometry is isomorphic to that of the original structure. Furthermore, we show that the symmetric reduct is in fact Hrushovski's construction for a symmetric ternary relation. This is easy, and so we immediately ask

\begin{question*}
Does Hrushovski's \emph{symmetric} construction admit a proper reduct with an isomorphic geometry?
\end{question*}

We answer this in the positive, also. Consider the following formula
\[S(x_1,x_2,x_3):= \exists x_4,x_5 \bigwedge_{i\in\set{1,2,3}} R(x_i,x_4,x_5) \wedge \bigwedge_{i\neq j} x_i \neq x_j\]
Call the $S$-reduct of Hrushovski's symmetric non-collapsed construction $M_s$, we prove the following

\begin{Cor*}\emph{\textbf{\ref{properReduct}}}  $M_s$ is a proper reduct of Hrushovski's symmetric non-collapsed construction. Namely, the relation $R$ cannot be recovered from the relation $S$.
\end{Cor*}

\begin{theorem*}\emph{\textbf{\ref{PGMandMs}}}
The pre-geometries of Hrushovski's non-collapsed structure and $M_s$ are isomorphic
\end{theorem*}
The theorem directly above implies that $M_s$ and the non-symmetric Hrushovski construction have isomorphic geometries. We have already stated that the geometry of the non-symmetric structure is isomorphic to the geometry of the symmetric structure and thus so is the geometry of $M_s$.

Proving these results comprises the bulk of this work. We do not know whether the structure $M_s$ is bi-interpretable with the original structure or whether it admits a proper reduct with an isomorphic geometry, but it seems to be a possible candidate for being prime.

\subsection{Overview of the paper}

%We first give an ad-hoc description of a Fra\"iss\'e-Hrushovski type amalgamation class. There may be such amalgamation classes that our description omits, but this will serve for the purposes of this paper.

Consider a class of finite (or finitely generated) structures $\CC$ with a sub-modular function $d_0:\CC\to \N$ (called pre-dimension). Let $\mathfrak{M}$ be the class of embeddings of a structure from $\CC$ into another structure from $\CC$. Distinguish the, so called, \emph{strong} embeddings $f_0\in \mathfrak M$, $f_0: A\to B$ such that $d_0(A') \geq d_0(A)$ for any embedding $f:A'\to B$ with $f_0\subseteq f$.   If $\CC$ with the class of strong embeddings is an amalgamation class (so, in particular, forms a category), we say -- for the purposes of this paper only -- that $(\CC, d_0)$ is a Fra\"iss\'e-Hrushovski amalgamation class. 
%Further discussion of this may be found at the beginning of section \ref{Mspregeometry}.

%$Hrushovski's construction is the generic construction of such an amalgamation class $\CC$. 
The restrictions imposed by the function $d_0$ on the class of strong embeddings assure that $T$, the theory of the Fra\"iss\'e limit of $(\CC, d_0)$, is superstable and supports, naturally, an independence relation, which -- on abstract grounds -- must be non-forking. Strong minimality is then obtained by imposing further restrictions on the class $\CC$ (known as "the collapse" stage of the construction).

In the original construction introduced by Hrushovski, $\CC$ was the class of finite structures in the language of a ternary relation $R$, with the special property that every substructure has no less points than relations. The function $d_0$ is then given by the difference between the number of points and the number of relations in the structure.

\smallskip
The results described in the previous sub-section all follow from the fact that $M_s$ is, too, the Fra\"iss\'e limit of its own age with resepct to a Hrushovski-type pre-dimension function. We observe, however, that there is no (non-trivial) bound on the number of relations a finite sub-structure of $M_s$ may have. Thus, instead of counting relations, our dimension function counts sizes of \emph{cliques} (with some minor technical modifications). In fact, adapting this function to Hrushovski's construction, it coincides with the original pre-dimension function. The technical heart of the paper is, thus, in the proof that $M_s$ is the Fra\"iss\'e-Hrushovski limit of its own age, with respect to this revised pre-dimension function, $\dos$.

This is obtained in several steps. First, we prove that $\CC_s$, the class of $S$-structures with hereditarily non-negative pre-dimension, is the age of $M_s$ (conclusion of Section \ref{age}). We then proceed to show that $M_s$ is the Fra\"iss\'e limit of $(\CC_s,\dos)$ and show that $M_s$ is a proper reduct (Section \ref{MsFraisse}). In Section \ref{Mstheory} we reap the fruits, axiomatising the Fraisse limit and recognizing the level of quantifier elimination of the theory. In Section \ref{Mspregeometry} we exploit further these results to study the geometry of $M_s$.

\smallskip
In sections \ref{defs}, \ref{SStructs} and \ref{age} we develop the tool box needed for the later sections. Among these, two may be worth mentioning already at this stage.
Unlike in Hrushovski's original structure, amalgamation in $\CC_s$ is not unique. As can be easily verified, a clique of size 3 can split (i.e., have two cliques extending it, whose union does not form a clique), but there is a bound (in fact, 3) on the number of times a clique may split. Thus, given two cliques $C_1,C_2\in \CC_s$ extending a common sub-clique $C\in \CC_s$, there is no unique ``free amalgam'' of $C_1$ with $C_2$ over $C$. 
% should the amalgam be one clique or two? And what if $C_2$ is not one clique but two distinct cliques both extending $C$? To which should we append $C_1$ -- to the first? to the second, to neither? And how do we know how many times should a given clique $C$ be counted? 
The notion of \emph{multiplicity} of a clique, introduced in Section \ref{defs}, provides the technical means for handling this problem. In section \ref{SStructs} the appropriate amalgam of $\CC_s$ is developed.

A technically more involved problem is related to witnesses: from the definition of the relation $S$ it follows, that in the structure $M$, given an $S$-clique $C$, there are witnesses $x_C,y_C$ such that $M\models R(z,x_C,y_C)$ for all $z\in C$. We call $x_C, y_C$ the \emph{witnesses} for  $C$. Using multiplicity, we may assume that each clique has a unique set of witnesses. Thus, the number and size of cliques (in a finite $M_s$-substructure $A$) gives an estimate on the number of $R$-relations in $A$: add all witnesses to $A$, then each pair of witnesses, $x_C,y_C$ contributes $|C|$ relations, and therefore $d_0(A)$ can be estimated by $\sum (|C|-2)$. However, this estimate is readily verified as an upper bound only in case $x_C, y_C$ are not in $A$. In Section \ref{age} we develop the technique of \emph{outsourcing} which allows us to assume that this is indeed the case. 

%As stated before, in Section \ref{symmetric} is the treatment of the symmetric reduct of Hrushovski's original construction. 
In order to keep the text more readable some of the more technical proofs were differed to Section \ref{appendix},  an appendix to this paper. \\

% It still remains to verify that everything we have done transfers properly to the collapsed structure. This indeed seems to be the case, under the assumption that the collapse is lenient enough (large enough $\mu$ values for $\mu$ as defined in \cite{Hns}).
% 
% Finally, it is also worth noting that we may readily generalise $M_s$ to Hrushovski constructions for a relation of arity greater than 3. For a relation of arity $n$, we adapt the relation $S$ such that the cliques will be of size $n$ and above and that the generic dimension of a clique is $(n-1)$. Everything should then work quite similarly.
\noindent\textbf{Acknowledgement}
In writing the paper we consulted, beside Hrushovski's original paper, \cite{Hns}, also M. Ziegler's exposition of Hrushovski's construction, \cite{Ziegler}. We would also like to thank D. Evans, whose questions helped initiating this project, and whose comments and suggestions contributed to its progress. This paper is part of the second author's M.Sc. thesis.

\section{Definitions} \label{defs}
We refer to Hrushovski's non-collapsed countable structure for a symmetric (i.e., invariant under permutation of the variables), ternary relation $R$ as $M$\cite{Hns}. The age of $M$ is denoted $\CC$. We assume, without loss of generality, that $R$ is such that if  $(x,y,z)\in R$ then $x,y,z$ are distinct elements. We let $A,B,C\dots$ denote elements in $\CC$, $N,P\dots$ will denote possibly infinite sub-structures of $M$.  As in Hrushovski's construction $d_0(A)=|A|-r(A)$ where $r(A):=\{\{x,y,z\}\subseteq A: (x,y,z)\in R(A)\}$. We set $d(A,B)=\min\{d_0(C):A\subseteq C\subseteq A\cup B\}$ and $A\leqslant B$ if $A$ is \emph{strong} (or \emph{self-sufficient}) in $B$, i.e., $d_0(A)=d(A,B)$. 
%  The pre-dimension and dimension functions $d,d_0$ as well as the relation $\leqslant$ are as defined in the original paper (adjusted to a symmetric relation). The letters $A,B,C,D$ will usually denote finite structures, the letter $N$ will denote a possibly infinite structure.

\begin{remark*}
As $R$ is a symmetric irreflexive relation, we may refer to relations in $R$ as both ordered triples and as sets of size three. i.e. $(x,y,z)\in R \Leftrightarrow \set{x,y,z}\in R$.
\end{remark*}

\begin{definition} Let \[S(x_1,x_2,x_3):= \exists x_4,x_5 \bigwedge_{i\in\set{1,2,3}} R(x_i,x_4,x_5) \wedge \bigwedge_{i\neq j} x_i \neq x_j\]
Note that the relation $S$ is also symmetric.
\end{definition}

This work will be concerned with a reduct of the structure $M$ to the relation $S$ which we will call $M_s$. We begin by introducing terminology and several facts regarding the relation $S$ in the context of $R$-structures, and $M$ in particular.

\begin{notation}
Given an $R$-structure $A$, we denote the interpretation of the relation $R$ in $A$ by $R_A =\setarg{(a_1,a_2,a_3)}{a_1,a_2,a_3\in A\text{ and } A\models R(a_1,a_2,a_3)}$. When discussing the structure $M$, we simply denote the interpretation as $R$.
\end{notation}

\begin{definition}\label{C(A)forR}
Let $N$ be an $R$-structure and $A\subseteq N$ a finite substructure of $N$. In the context of the superstructure $N$:

We say $\bar{C} = (C,\set{b_1,b_2})$ is an \emph{$S$-clique in $A\subseteq N$} if $C\subseteq A$, $|C|\geq 3$, $b_1,b_2\in N\setminus C$ and $(c,b_1,b_2)\in R_N$ for any $c\in C$. We say $\bar{C}$ is \textbf{maximal} if also $(a,b_1,b_2)\notin R_N$ for any $a\in A\setminus C\cup\set{b_1,b_2}$. We say $b_1,b_2$ are \emph{the witnesses} of $\bar{C}$ and $C$ is \emph{the universe} of $\bar{C}$.

We say $C\subseteq A$ is the \emph{universe of an (maximal) $S$-clique in $A\subseteq N$} if there exist $b_1,b_2\in N\setminus C$ such that $(C,\set{b_1,b_2})$ is an (maximal) $S$-clique in $A\subseteq N$.

Define $C(A,N) = \setarg{\bar{C}}{\bar{C}=(C,\set{x^C,y^C}) \text{ is a maximal }S\text{-clique in }A\subseteq N}$

In cases where it is obvious what the superstructure $N$ is (or $A$ itself is the superstructure), we simply omit it and define: \emph{$S$-clique in $A$}, \emph{universe of an $S$-clique in $A$} and $C(A)$ without reference to $N$. Where applicable, the superstructure will be $M$.
\end{definition}

\begin{remark}
We state several conventions we will use throughout the text:

\begin{itemize}
\renewcommand{\labelitemi}{$-$}
\item
We will often abuse notation and think of a given maximal $S$-clique $\bar{C}\in C(A)$, where $A$ is a given $R$-structure, both as an ordered pair $\bar{C} = (C, \set{x,y})$ and as its universe, $C$. Ofttimes we omit the bar and simply refer to the clique by $C$, its universe.

\item
Unless otherwise stated, when referring to a clique we assume it is maximal.

\item
We call $C(A)$ the \emph{enhanced-$S$-diagram} of $A$, as it contains the information of exactly which triples of elements from $A$ are in the relation $S$ and how the relations are witnessed.

\item
We assume that $d(A)$ is an integer for any finite $R$-structure $A$. This implies that $C(A)$ is finite for any finite $R$-structure $A$.
\end{itemize}
\end{remark}

\begin{definition}
Let $C$ be a set and $x,y$ given elements. Define:
\\$RW(C,\set{x,y}) = \setarg{(c,x,y)}{c\in C}$

\smallskip
Let $A\subseteq N$ be an $R$-structure. Let $C\in C(A)$ with $C = (C,\set{x,y})$. Define:
\\$wit(C) = \set{x,y}$
\\$RW(C) = RW(C,wit(C))$

\smallskip
We call $RW(C)$ the set of \emph{generating relations} of the clique $C$ and say that a relation $r\in RW(C)$ is \emph{used by the clique $C$} or is a \emph{generator of $C$}.
\end{definition}

\begin{definition}
Let $A\subseteq N$ be an $R$-structure. Let $C_1,C_2\in C(A)$. If $RW(C_1)\cap RW(C_2)\neq \emptyset$ we say $C_1$ and $C_2$ \emph{share a relation} or \emph{use the same relation}. If $r\in RW(C_1)\cap RW(C_2)$ we say $r$ is \emph{shared} or \emph{used} by $C_1$ and $C_2$. The cliques $C_1$ and $C_2$ need not necessarily be maximal for the terminology to apply.
\end{definition}

\begin{obs}\label{sharingRelation}
Say $r = \set{a,b,c}$ is shared by $C_1 = (C_1,\set{x_1,y_1})$ and $C_2 = (C_2,\set{x_2,y_2})$, distinct maximal cliques in some $R$ structure $A\subseteq N$.

As $r\in RW(C_i)$, it must be that $wit(C_1)\cup wit(C_2)\subseteq r$. As $|r| = 3$ this means $wit(C_1)\cap wit(C_2) \neq \emptyset$. So it must be that $|wit(C_1)\cap wit(C_2)| \geq 1$, meaning $C_1$ and $C_2$ have a common witness. without loss of generality $x_1 = x_2 \overset{def}{=} x$. Since $C_1$ and $C_2$ are different maximal cliques, it must be that $wit(C_1)\neq wit(C_2)$ and so the elements $x,y_1,y_2$ are distinct.

So $\set{x,y_1,y_2}\subseteq r$ and so $r = \set{x,y_1,y_2}$. As $r$ is determined by the two cliques that share it, there is at most one relation shared by two given cliques and it is the union of their pairs of witnesses. Also, since $r\in RW(C_1)$ it must be that $y_2\in C_1$ and similarly $y_1\in C_2$.
\end{obs}

\begin{definition}
Let $A\subseteq N$ be $R$-structures. For a clique $(C,\set{x,y})$ (not necessarily maximal) in $A$, we say that $x$ (or $y$) is an external witness if $x\notin A$ (or $y\notin A$). If that is the case, we say $C$ is \emph{externally witnessed} or that its pair of witnesses is \emph{external}.
\end{definition}

\begin{obs}\label{addinExternalWitnesses}
Let $A\subseteq N$ be $R$-structures. Say $C\subseteq A$ is the universe of a clique (not necessarily maximal) in $A$ witnessed externally by some $\set{b_1,b_2}$. Then $d_0(b_1,b_2/A) \leq 2-|C|$. This is true because $RW(C,\set{b_1,b_2})\cap R_A = \emptyset$ and thus $r(b_1,b_2/A) \geq |RW(C,\set{b_1,b_2})| = |C|$, on the other hand $|\set{b_1,b_2}\setminus A| \leq 2$ and so we have $d_0(b_1,b_2/A) \leq 2-|C|$. In case exactly one of $b_1,b_2$ is external to $A$ we have $d_0(b_1,b_2/A) \leq 1-|C|$ by the same argument.

The observation holds in case $|C| < 3$ as well.
\end{obs}

We introduce the notion of multiplicity. This will be a crucial notion once we define abstract $S$-structures. We also show that it is natural to discuss multiplicity in the setting of the relation $S$ in preparation for the definition of pure $S$-structures (with no given underlying $R$ structure).

\begin{definition} \label{multiplicity}
For $A$ an $R$-structure and for $C\subseteq A$ the universe of an $S$-clique (not necessarily maximal) in $A$ - define $m_A(C)$, the multiplicity of $C$ in $A$, to be the number of distinct pairs witnessing that $C$ is a clique in $A$. For $C$ not the universe of an $S$-clique (not necessarily maximal) define $m_A(C) = 0$. When the structure in discussion is $M$, we may simply write $m(C)$.
\end{definition}

\begin{lemma}\label{multiplicityMaximum} Say $C$ is the universe of an $S$-clique (not necessarily maximal) in $M$. Then:
\begin{enumerate}
\item If $|C| > 4$ then $m(C) = 1$
\item If $|C| = 4$ then $1\leq m(C) \leq 2$
\item If $|C| = 3$ then $1\leq m(C) \leq 3$
\end{enumerate}
\end{lemma}

\begin{proof} Since $C$ is the universe of an S-clique, by definition $m(C)\geq 1$. Let $\otn{p}{m(C)}$ be distinct pairs of witnesses for $C$. Note that for any $1\leq i\leq m(C)$ we have $p_i\cap C = \emptyset$ by definition of $S$. It cannot be that $RW(C,p_i)\cap RW(C,p_j) \neq \emptyset$ for $i\neq j$, for then by Observation \ref{sharingRelation}, $p_i\cap C \neq \emptyset$, which we said cannot happen. So by iterating \ref{addinExternalWitnesses} we have:
\\$0\leq d_0(C,\otn{p}{m(C)})\leq d_0(C) + d_0(\otn{p}{m(C)}/C) \leq |C| - m(C)\cdot (|C|-2)$ 
\\So $m(C)\cdot (|C|-2) \leq |C| $  and the lemma is evident.
\end{proof}

The proof of the following fact is very tedious and uninteresting. It can be found in the appendix (Corollary \ref{morethanfiveclique}), but it is highly advised to avoid it. We will use it nonetheless.

\begin{fact}\label{noUnwitnessedCliques}
Let $A$ be an $R$-structure with $\emptyset \leqslant A$. If $C\subseteq A$ with $|C| > 4$ and $S(c_1,c_2,c_3)$ for all $c_1,c_2,c_3\in C$ distinct, then $C$ is the universe of a clique (not necessarily maximal).
\end{fact}

We will now consider and define structures in the setting of the relation $S$ without the relation $R$. We begin by showing that $S$ \emph{knows} the multiplicity of subsets of $M$.

\begin{prop}\label{geqMultiplicityPredicate} Assume $m(C)<\infty$ for any $C\subseteq M$, then the predicate "$C$ is an $S$-clique of size k with $m(C) \geq n$" can be written in $M$ using only the relation $S$.
\end{prop}

\begin{proof}
In $M$, a clique has multiplicity greater than $n$ if and only if it has $n+1$ extensions, such that the union of any two of these extensions is not a clique. We will shortly prove this fact, and use it in order to indicate the multiplicity of a clique in $M$ with a first order formula.

\medskip
Denote \[\psi_k(\otn{v}{k}):= \bigwedge_{1\leq i<j<l\leq k} S(v_i,v_j,v_l)\land \bigwedge_{1\leq i<j\leq k} v_i\neq v_j\]
\\$M \models\psi_{|\bar{a}|}(\bar{a})$ means that $\bar{a}$ as a set is a clique of size $|\bar{a}|$ under the relation $S$.

For $k\geq 3$, $n\geq 1$, consider the formula $\varphi_{k,n}(\otn{v}{k}):=$
\[\exists \otn{\bar{x}}{n} \left( \psi_k(\otn{v}{k})
\land\bigwedge_{1\leq i \leq n}\psi_{k+2}(\otn{v}{k},\bar{x_i})
\land\bigwedge_{1\leq i< j\leq n}\neg\psi_{k+4}(\otn{v}{k},\bar{x_i},\bar{x_j}) \right)\]

Where each $\bar{x}_i$ is a pair of distinct variables. We take $\bar{x}$ to be a tuple rather than a single variable in order to use Fact \ref{noUnwitnessedCliques} with $\set{\otn{v}{k},\bar{x}_i}$.

This formula will express the fact that its arguments are a clique of size $k$ with multiplicity at least $n$.

\noindent\textbf{Claim.} Let $\bar{a}\in M^k$. Then $M\models \varphi_{k,n}(\bar{a}) \iff m(\bar{a}) \geq n$
\begin{proof}
\underline{$\Leftarrow$}: Say $m(\bar{a}) = m \geq n$, then there are exactly $\otn{p}{m}$ distinct pairs of witnesses for $\bar{a}$. By saturation of $M$ there are pairs of elements $\otn{\bar{x}}{n} = (x^1_1,x^1_2),...,(x^n_1,x^n_2)$ such that
$M\models R(x^i_t,p_i)\land\bigwedge_{\substack{j\neq i\\1\leq j\leq m}}\neg R(x^i_t,p_j)$ for $1\leq i\leq n$, $t\in\set{1,2}$. Fix $\otn{\bar{x}}{n}$.

So $M\models \psi_k(\bar{a}) \land\bigwedge_{1\leq i \leq n}\psi_{k+2}(\bar{a},\bar{x_i})$. If $M\models \psi_{k+4}(\bar{a},\bar{x}_i,\bar{x}_j)$ for some $1\leq i<j\leq n$ then by Fact \ref{noUnwitnessedCliques}, there is a pair $p\subseteq M$ witnessing $\bar{a},\bar{x}_i,\bar{x}_j$. Then in particular $p$ witnesses $\bar{a}$ and therefore $p=p_r$ for some $1\leq r\leq m$. Without loss of generality $r\neq j$ and so this is a contradiction to our choice of $\bar{x}_j$. Thus also $M\models\bigwedge_{\substack{j\neq i\\1\leq j\leq m}}\neg\psi_{k+4}(\bar{a},\bar{x}_i,\bar{x}_j)$ and $M\models \varphi_{k,n}(\bar{a})$

\underline{$\Rightarrow$}: Say $M\models \varphi_{k,n}(\bar{a})$. Fix some $\otn{\bar{x}}{n}\in M^2$ as guaranteed in $\varphi_{k,n}(\bar{a})$. By Fact \ref{noUnwitnessedCliques}, there are pairs $\otn{p}{n}$ witnessing $(\bar{a},\bar{x}_1),...,(\bar{a},\bar{x}_n)$ respectively. It cannot be that $p_i = p_j$ for $1\leq i<j\leq n$, because then we would have $M\models\psi_{k+4}(\bar{a},\bar{x}_i,\bar{x}_j)$. So the pairs $\otn{p}{n}$ are distinct. Since each pair $p_i$ is a pair witnessing $\bar{a}$, $m(\bar{a}) \geq n$.
\end{proof}

So the claim shows that the proposition is true for $k\geq 3$ and $n\geq 1$. But the rest is trivial, as a set of size $k<3$ always has multiplicity $0$ and any set at all has multiplicity equal or greater to zero.
\end{proof}

\begin{Cor}
The predicate "$C$ is an $S$-clique of size $k$ with $m(C) = n$" can be written in $M$ using only the relation $S$.
\end{Cor}

\begin{proof}
Lemma \ref{multiplicityMaximum} assures us the for any $C\subseteq M$ we have $m(C)<\infty$ and so by Proposition \ref{geqMultiplicityPredicate} we can write  $\varphi_{k,n}(C) :=$ "$C$ is an $S$-clique of size k with $m(C) \geq n$" in $M$ using only $S$. Now the predicate we seek is simply $\phi_{k,n}(C) := \varphi_{k,n}(C)\wedge \neg\varphi_{k,n+1}(C)$
\end{proof}

\begin{remark}\label{languageExpansionRemark}
We may now definably expand any $R$ structure $A$ with $\emptyset\leqslant A$ to a two-sorted structure in the language $\calL'$ with relation symbols $\set{R,S}$, n-ary function symbols $\setarg{m_n}{n\in \N}$ and constants $\set{0,1,2,3}$. The new sort will be the set of new constants and the function $m_n: M^n\rightarrow \set{0,1,2,3}$ will be the multiplicity function for sets of size $n$. The reduct $M_s$ which we spoke of is infact the reduct of this definable expansion of $M$, where we forget the relation $R$. As we have just shown, the functions are definable from $S$ and so we have not infact added any new data to a reduct where we only remember $S$.

This technicality cannot be skipped, as we will later (when discussing the theory of $M_s$) need the dimension of finite pure $S$-structures to be a first order property. For multiplicity to be definable using the first order formula presented above, a structure must be sufficiently saturated, which will not always be the case. In order to allow for readability, we will not use this notation, and instead simply acknowledge that we may use multiplicity in a first order setting.
\end{remark}

To better understand the following definition, it is best to think of the set $P$ as $C(A)$.

\begin{definition}
Let $A$ be a set. We say $\mult$ is a \emph{pre-multiplicity} function if $\mult:P\rightarrow \set{1,2,3}$ for some $P\subseteq [A]^{\geq 3}$, where we remind that $[A]^{\geq 3} = \setarg{X\subseteq A}{|X| \geq 3}$.

\medskip
We define $m:P(A)\rightarrow \N\cup\set{\infty}$, the \emph{associated multiplicity function} of $\mult$, as follows:
\\- $m(C) = 0$ for all $C\subseteq A$ with $|C| < 3$
\\- For $C\subseteq A$ with $|C|\geq 3$ let $\set{C_i}_{i\in I}$ distinct be all the sets in $P$ with $C\subseteq C_i$, then $m(C) = \sum_{i=I} \mult(C_i)$ (where if the sum is infinite, $m(C) = \infty$) 

\medskip
We say $\mult$ is a \emph{valid} pre-multiplicity function if:
\\- For $C\subseteq A$ with $|C|\geq 3$, as in Lemma \ref{multiplicityMaximum}, $m(C)\leq 1 + \frac{2}{(|C|-2)}$.
\end{definition}

\begin{remark}\label{UniqueDefineSstructure}
A valid finite pre-multiplicity function $\mult_\mathcal{A}$ can be recovered from its associated multiplicity function $m_\mathcal{A}$, and $C(\mathcal{A})$ may be recovered from $\mult_\mathcal{A}$. So in order to uniquely define a finite $S$-structure we only need its pre-multiplicity or multiplicity function.
\end{remark}

\begin{definition} An $S$-structure $\mathcal{A}$ is the following data: A set of points $A$ with $C(\mathcal{A})\subseteq [A]^{\geq 3}$ the set of declared maximal cliques in $A$ (the cliques here are merely sets, no witnesses specified), and a \emph{valid} pre-multiplicity function $\mult_\mathcal{A}:C(\mathcal{A})\rightarrow \set{1,2,3}$.

\medskip
We define $m_\mathcal{A}:P(A)\rightarrow \set{0,1,2,3}$, the multiplicity function of $\mathcal{A}$, to be the associated multiplicity function of $\mult_\mathcal{A}$ 
\end{definition}

\begin{definition}
Let $\mathcal{A}$ and $\mathcal{B}$ be $S$-structures. We say $\mathcal{A}$ is a substructure of $\mathcal{B}$ and denote $\mathcal{A}\subseteq \mathcal{B}$ if:
\begin{itemize}
\itemlabel{$-$}
\item $A\subseteq B$
\item $C(\mathcal{A}) = \setarg{C\cap A}{C\in C(\mathcal{B})}$
\item (Additivity) For any $C\in C(\mathcal{A})$
\[\mult_\mathcal{A}(C) = \sum_{\substack{C'\in C(\mathcal{B})\\ C'\cap A = C}}\mult_\mathcal{B}(C')\]
\end{itemize}

For $A'\subseteq A$, we define an associated $S$-structure $\mathcal{A'}\subseteq \mathcal{A}$ with universe $A'$ in the natural way.
\end{definition}

\begin{obs*}
Let $\mathcal{A}\subseteq \mathcal{B}$ be finite where $A$ is the universe of $\mathcal{A}$. Then $m_\mathcal{A}(C) = m_\mathcal{B}(C)$ for any $C\subseteq A$. Then by Remark \ref{UniqueDefineSstructure} $\mathcal{A}\subseteq \mathcal{B}$ iff $m_\mathcal{A} = (m_\mathcal{B}\upharpoonright P(A))$.

We say the $S$-structures $\set{\mathcal{A}}_{i\in I}$ agree on the multiplicity, if whenever $C$ is a subset of the universe of $\mathcal{A}_i,\mathcal{A}_j$, then $m_{\mathcal{A}_i}(C) = m_{\mathcal{A}_j}(C)$. In case all structures in discussion agree on the multiplicity function, we may use the notation $m$ with no subscript.
\end{obs*}

\begin{definition} Let $\mathcal{A}\subseteq \mathcal{N}$ be $S$-structures with $\mathcal{A}$ finite. Define:
\\$\dos(\mathcal{A}) = |A| - \sum_{C\in C(\mathcal{A})}\mult_{\mathcal{A}}(C)\cdot (|C|-2)$.
\\$\ds(\mathcal{A},\mathcal{N}) = \min\setarg{\dos(\mathcal{D})}{\mathcal{A}\subseteq \mathcal{D} \subseteq \mathcal{N}\text{, }\mathcal{D}\text{ finite}}$ 
\\$\mathcal{A} \leqs \mathcal{N}$ if $\dos(\mathcal{A}) = \ds(\mathcal{A},\mathcal{N})$. We say $\mathcal{A}$ is \emph{strong} or \emph{self-sufficient} in $\mathcal{N}$.
\medskip
\\We expand the definition to infinite $\mathcal{A}$ and say $\mathcal{A}\leqs\mathcal{N}$ if for any $\mathcal{D}\subseteq\mathcal{A}$ finite, $\mathcal{D}\leqs\mathcal{A}$ implies $\mathcal{D}\leqs\mathcal{N}$
\end{definition}

\begin{remark*}
For any structure $\mathcal{N}$ appearing throughout the text, we assume that $d_s(\emptyset, \mathcal{N})$ exists.
\end{remark*}

\begin{definition}
Let $\mathcal{A}$ and $\mathcal{N}$ be $S$-structures with universes $A$ and $N$, respectively. Let $f:A\rightarrow N$. Denote $A'=f[A]\subseteq N$ and denote $\mathcal{A}'$ the naturally defined $S$-structure on the set $A'$ as a substructure of $\mathcal{N}$. We say $f$ is an embedding of $\mathcal{A}$ into $\mathcal{N}$ if:
\begin{enumerate}
\item $f$ is injective.
\item $C(\mathcal{A}') = \setarg{f[C]}{C\in C(\mathcal{A})}$
\item $m_{\mathcal{A}'}(f[C]) = m_\mathcal{A}(C)$ for any $C\subseteq A$.
\end{enumerate}
\end{definition}

\begin{definition}
Let $\mathcal{A}$ and $\mathcal{B}$ be $S$-structures with universes $A$ and $B$, respectively. Let $f:A\rightarrow B$. We say $f$ is an isomorphism of $S$-structures between $\mathcal{A}$ and $\mathcal{B}$ if $f$ is a bijective embedding of $\mathcal{A}$ into $\mathcal{B}$.
\end{definition}

\begin{definition}
Let $\mathcal{A}$ and $\mathcal{N}$ be two $S$-structures and $f$ an embedding of $\calA$ into $\calN$. Denote $f(\calA)\subseteq \calN$ the image of $\calA$ under $f$. We say $f$ is a \emph{strong embedding} if $f(\mathcal{A})\leqs \mathcal{N}$.
\end{definition}

We now link abstract $S$-structures and $R$-structures that give rise to an $S$-structure and observe several facts that will be useful later.

\begin{definition}
Let $A\subseteq N$ be an $R$-structure. We define its associated $S$-structure $\mathcal{A}$ with respect to $N$ to be the structure with:
\begin{itemize}
\itemlabel{$-$}
\item
The universe $A$
\item
$C(\mathcal{A}) = \setarg{C\subseteq A}{C\text{ is the universe of some clique in }C(A)}$
\item
The multiplicity function $m_\mathcal{A}$ is defined to be the restriction of $m_N$ (as defined in Definition \ref{multiplicity}) to $P(A)$.
\end{itemize}

\medskip
In most cases we will assume implicitly that $N$ is known and so we will speak of the $S$-structure associated with $A$.
\end{definition}

\begin{remark*}
In the above definition, instead of defining a multiplicity function, we could have equally defined a pre-multiplicity function on $C(\mathcal{A})$ in the natural way: $\mult_{\mathcal{A}}(C) = |\setarg{\bar{C}\in C(A)}{C\text{ is the universe of }\bar{C}}|$

The multiplicity function associated with this pre-multiplicity function is the same as the one defined above
\end{remark*}

\begin{remark}
We leave as an exercise to the reader the fact that if $A\subseteq B$ are $R$-structures with associated $S$-structures $\mathcal{A}$ and $\mathcal{B}$ respectively, then $\mathcal{A}\subseteq \mathcal{B}$.
\end{remark}

\begin{definition}\label{bijectionFromRtoS}
Say $A$ is an $R$-structure with associated $S$-structure $\mathcal{A}$. 
Let $(\otn{\bar{C}}{n})$ be an enumeration of $C(A)$. Consider now the vector of universes $\mathcal{U} = (\otn{C}{n})$ where $C_i$ is the universe of $\bar{C}_i$. The universe $C_i$ appears in $\mathcal{U}$ exactly $\mult_\mathcal{A}(C_i)$ times for any $1\leq i\leq n$. We call an enumeration such as $\mathcal{U}$ a \emph{full-enumeration of $C(\mathcal{A})$}.

For a general $S$-structure $\mathcal{A}$, we say a vector of universes $\mathcal{U} = (\otn{C}{n})$ is a full-enumeration of $K\subseteq C(\mathcal{A})$ if it only contains elements from $K$ and each $C\in K$ appears in $\mathcal{U}$ exactly $\mult_\mathcal{A}(C)$ times.
\end{definition}

A full-enumeration of $C(\mathcal{A})$ for an associated $S$-structure $\mathcal{A}$ will aid us with counting arguments, as it induces a correspondence between any $C\in C(\mathcal{A})$ and the $\mult_\mathcal{A}(C)$ cliques in $C(A)$ with universe $C$.

\begin{definition}
Let $A\subseteq B$ be $R$-structures with associated $S$-structures $\mathcal{A}$ and $\mathcal{B}$ respectively. We define:
\\$\dos(A) = \dos(\mathcal{A})$
\\$A\leqs B$ iff $\mathcal{A}\leqs \mathcal{B}$
\end{definition}

\begin{definition}
Let $N$ be an $R$-structure. Let $x,y\in N$ be two distinct elements.

\noindent For any subset $A\subseteq N$ define $\witclq{x,y}{A} = \setarg{a\in A}{(a,x,y)\in R_N}$

\noindent If $\witclq{x,y}{A}$ is the universe of a clique we identify it with the clique $\clqwit{\witclq{x,y}{A}}{x,y}$.

\noindent For $K$ a set of cliques define $K^A = \setarg{\witclq{x,y}{A}}{\clqwit{C}{x,y}\in K \text{ and } |\witclq{x,y}{A}| \geq 3}$
\end{definition}

\begin{definition}
Let $A,B\subseteq D$ be $R$-substructures. For $C$, the universe of a maximal clique in $A$, denote $C^{B(A)} = \setarg{C'\in C(B\cup A)}{C'\cap A = C}$, the set of maximal cliques in $B\cup A$ extending $C$.

\noindent Define:
\\$C(B/A) = \setarg{C\in C(B\cup A)}{|C\cap A| < 3}$
\\$C^B(A) = \bigcup_{C\in C(A)} C^{B(A)}$

$C(B/A)$ is the set of cliques in $B$ which are not extensions of cliques from $A$. The set $C^B(A)$ is the collection of maximizations of cliques in $C(A)$ to the structure $B\cup A$. Take note that $C(B\cup A) = C^B(A)\amalg C(B/A)$.

The definitions are identical for $S$-structures.
\end{definition}

We conclude this section with an alternate way of defining $\dos(A)$ which we will use in the setting of $R$-structures.

\begin{definition}
Let $A\subseteq N$ be an $R$-structure. Let $K\subseteq C(A)$. Define:
\\$Clq_K(a) = |\setarg{C\in K}{a\in C}|$
\\$Clq(a,A) = Clq_{C(A)}(a)$
\\$i(A) = \sum_{a\in A} Clq(a,A) = \sum_{C\in C(A)} |C|$
\end{definition}

\begin{lemma}\label{IdenticalD0s}
$\dos(A) = |A| + 2|C(A)|-i(A)$
\end{lemma}

\begin{proof}
Denote $\mathcal{A}$ the $S$-structure associated with $A$. As in Observation \ref{bijectionFromRtoS}, there is a clear correspondence between a clique $C\in C(\mathcal{A})$ and the $\mult_\mathcal{A}(C)$ cliques in $C(A)$ with universe $C$. Thus we have:
\begin{align*}
|A| + 2|C(A)| - i(A) &= |A| + 2|C(A)| - \sum_{C\in C(A)}|C|
\\&= |A| -\sum_{C\in C(A)} (|C|-2)
\\&= |A|-\sum_{C\in C(\mathcal{A})} \mult_\mathcal{A}(C)\cdot (|C|-2)
\\&= \dos(\mathcal{A}) = \dos(A)
\end{align*}
\end{proof}

\begin{remark*}
From this point onwards we identify an $S$-structure $\mathcal{A}$ with its universe $A$. In case $\mathcal{A}$ is the $S$-structure associated with the $R$-structure $A$, we think of $A$ dually as an $R$-structure and as an $S$-structure.
\end{remark*}

We will proceed to examine the reduct $M_s$ of $M$ in which we forget the relation $R$ and remember only $S$ (as discussed in Remark \ref{languageExpansionRemark}). The main goal of this work is to first show that $M_s$ is a proper reduct. Namely, that $R$ cannot be retrieved from $S$. Once that is achieved, we would like to know what is the d-geometry of $M_s$.

\section{Properties of $S$-structures}\label{SStructs}

\begin{definition}
For $A,B \subseteq M_s$ denote $\dosover{B}{A} = \dos(B\cup A) - \dos(A)$, the pre-dimension of $B$ over $A$.
\end{definition}

\begin{obs*}
\begin{enumerate}[(i)]
\item $A\leqs B$ if and only if for all $X\subseteq B$ we have $\dosover{X}{A} \geq 0$.
\item If $A\leqs B$ then for any $X\subseteq B$ we have $A\leqs X\cup A$.
\item Let $A\subseteq B$. For a maximal clique $C\in C(B)$ note that if $|C\cap A| \geq 3$ then $C\cap A$ is a maximal clique in $C(A)$. We say $C$ extends $C\cap A$.
\end{enumerate}
\end{obs*}

\begin{definition}
For a set $A$ define $|A|_* = max\set{0,|A|-2}$.

Define also for $A$ an $S$-structure: $s(A) = \sum_{C\in C(A)}\mult_A(C)\cdot |C|_*$
\end{definition}

\begin{obs*}\*\begin{enumerate}
\item
Let $\fullenum{A}$ be a full enumeration of $C(A)$ for some $S$-structure $A$. Then $s(A) = \sum_{C\in \fullenum{A}}|C|_*$
\item
An alternative way to express the predimension of $A$ is $\dos(A) = |A| - s(A)$.
\end{enumerate}
\end{obs*}

\begin{lemma}\label{dosFreeAmalgam} Let $N$ be an $S$-structure. Let $B_1,B_2 \subseteq N$ be substructures of $N$. Denote $A = B_1\cap B_2$ and $D = B_1\cup B_2$. Then $\dosover{B_2}{A} \geq \dosover{B_2}{B_1}$.

Furthermore, if in addition $C(D/A) = C(B_1/A)\cup C(B_2/A)$ then equality holds and $\dosover{B_2}{A} = \dosover{B_2}{B_1}$.
\end{lemma}

\begin{proof}
Note that $B_2\setminus A = D \setminus B_1$ and so $|B_2| - |A| = |B_2\setminus A| = |D \setminus B_1| = |D| - |B_1|$. Then:
\begin{align*}
\dosover{B_2}{A} &\geq \dosover{B_2}{B_1} &\iff
\\|B_2| - |A| + s(A) - s(B_2) &\geq |D| - |B_1| + s(B_1) - s(D) &\iff
\\s(A)+s(D)&\geq s(B_1)+s(B_2) &\iff
\\s(D)\geq s(B_1&)+s(B_2)-s(A)
\end{align*}

\medskip
We prove the last inequality.

Fix $\fullenum{D}, \fullenum{B_1}, \fullenum{B_2},\fullenum{A}$ full enumerations of $C(D),C(B_1),C(B_2)$ and $C(A)$ respectively. It will now be enough to prove that:
\[\sum_{C\in\fullenum{D}}|C|_* \geq \sum_{C\in\fullenum{B_1}}|C|_* + \sum_{C\in\fullenum{B_2}}|C|_* - \sum_{C\in\fullenum{A}}|C|_*\]

For $I\in \set{B_1,B_2,A}$, fix $\sigma_I:\fullenum{D}\rightarrow\fullenum{I}\cup\set{\emptyset}$ surjective, such that
\[\sigma_I(C) = \left\{
	\begin{array}{ll}
		C\cap I & \text{if } C\cap I\in C(I)\\
		\emptyset & \text{otherwise}
	\end{array}
\right.\]
Such surjections exist by additivity of the pre-multiplicity function.

We observe several facts:
\begin{enumerate}
\item
By definition, $|C\cap I|_*=|\sigma_I(C)|_*$ for any $I\in \set{B_1,B_2,A}$ and $C\in \fullenum{D}$

\item Recall the equality $|C_1\cup C_2| = |C_1|+|C_2| - |C_1\cap C_2|$. Note that the inequality $|C_1\cup C_2|_* \geq |C_1|_*+|C_2|_* -|C_1\cap C_2|_*$ is always true. Equality holds in case $|C_1\cap C_2| \geq 2$ or $C_i = \emptyset$ for some $i\in\set{1,2}$.

\item
Let $C\subseteq D$, then $C = (C\cap B_1)\cup (C\cap B_2)$. By the previous item $|C|_* \geq |C\cap B_1|_* + |C\cap B_2|_* - |C\cap A|_*$.

\item
Let $C\in \fullenum{D}$, then by combining items $1$ and $3$ we get $|C|_* \leq |\sigma_{B_1}(C)|_* + |\sigma_{B_2}(C)|_* - |\sigma_A(C)|_*$. If $C\in C^D(A)\cup C(B_1/A)\cup C(B_2/A)$ then by item 2, equality holds.
\end{enumerate}

It now only remains to follow with the computation:
\begin{align*}
\sum_{C\in\fullenum{D}}|C|_* &\geq \sum_{C\in\fullenum{D}}(|\sigma_{B_1}(C)|_* + |\sigma_{B_2}(C)|_* - |\sigma_A(C)|_*) =
\\&= \sum_{C\in\fullenum{D}}|\sigma_{B_1}(C)|_* + \sum_{C\in\fullenum{D}}|\sigma_{B_2}(C)|_* - \sum_{C\in\fullenum{D}}|\sigma_A(C)|_* =
\\&=\sum_{C\in\fullenum{B_1}}|C|_* + \sum_{C\in\fullenum{B_2}}|C|_* - \sum_{C\in\fullenum{A}}|C|_*
\end{align*}

This proves the first part of the lemma.

If $C(D/A) = C(B_1/A)\cup C(B_2/A)$, then $C(D) = C^D(A)\cup C(B_1/A)\cup C(B_2/A)$. As observed in item $4$, this means that $|C|_* = |\sigma_{B_1}(C)|_* + |\sigma_{B_2}(C)|_* - |\sigma_A(C)|_*$ for any $C\in C(D)$. Thus, the first inequality in the above computation is in fact an equality. This proves $s(D)= s(B_1)+s(B_2)-s(A)$ and consequently the second part of the lemma.

\end{proof}

\begin{lemma}\label{StrongSubset} Let $A \leqs N$, then $X\cap A\leqs X$ for any $X\subseteq N$.
\end{lemma}

\begin{proof}
Assume such an $X$ is given. Choose an arbitrary finite $X\cap A \subseteq Y\subseteq X$. We wish to show $\dosover{Y}{X\cap A} \geq 0$. Note that $Y\cap A = X\cap A$ and so we in fact need to show $\dosover{Y}{Y\cap A} \geq 0$. By taking $B_1$ to be $A$ and $B_2$ to be $Y$, by Lemma \ref{dosFreeAmalgam} we have exactly $\dosover{Y}{Y\cap A} \geq \dosover{Y}{A} = \dosover{Y\cup A}{A}$. Since $A\leqs N$ we have $\dosover{Y\cup A}{A} \geq 0$ and thus $\dosover{Y}{Y\cap A} \geq 0$. The set $Y$ was arbitrary and so $X\cap A \leqs X$.
\end{proof}

\begin{Cor}\label{leqsTransitivity} The relation $\leqs$ is transitive.
\end{Cor}

\begin{proof}
Let $A\leqs B\leqs N$. In case $A$ or $B$ are infinite, the proof is by definition, so assume both are finite. Let $X\subseteq N$ be finite. By using Lemma \ref{StrongSubset} twice we have $X\cap A\leqs X\cap B$ and $X\cap B\leqs X$ and so $\dosover{X}{X\cap A} = \dosover{X}{X\cap B} + \dosover{X\cap B}{X\cap A} \geq 0$. Because $X$ is arbitrary we have $A\leqs C$.

\end{proof}

\begin{Cor}\label{intersectionOfStrongIsStrong}
Let $A_i\leqs N$ ($A_i$ not necessarily finite) for $i\in I$, then $\bigcap_{i\in I} A_i\leqs N$
\end{Cor}

\begin{proof}
We first prove the case where $A_i$ is finite for any $i\in I$:

Because any $A_i$ is finite, it can intersect with sets in only finitely many ways and so we may assume $I$ is finite. It is then enough to show the claim for $I = \set{1,2}$ and continue by induction.

Let $A_1,A_2\leqs N$. By Lemma \ref{StrongSubset} and $A_1\leqs N$ we have $A_1\cap A_2\leqs N\cap A_2 = A_2$. By $A_2\leqs N$ and transitivity we have $A_1\cap A_2\leqs N$ as desired.

\bigskip
We proceed to prove that $\bigcap_{i\in I} A_i\leqs N$ when $\set{A_i}_{i\in I}$ are possibly infinite:

Denote $\mathbb{A} =\bigcap_{i\in I} A_i$. Let $B\leqs\mathbb{A}$, we show that $B\leqs N$. Choose $B\subseteq B_i\subseteq A_i$ finite such that $\dos(B_i) = d_s(B,A_i)$. So $B_i\leqs A_i$ and thus $B_i\leqs N$. Denote $\mathbb{B} = \bigcap_{i\in I} B_i$. By the previous case, $\mathbb{B} \leqs N$. Because $B\subseteq \mathbb{B} \subseteq \mathbb{A}$ and $B\leqs \mathbb{A}$, we have $B \leqs \mathbb{B}$. By transitivity, $B\leqs N$.
\end{proof}

\begin{definition}
For $A\subseteq N$ $S$-structures we say $\bar{A}$ is an $S$-closure of $A$ if it is a smallest set (under inclusion) with $\bar{A}\leqs N$. By the above corollary there is exactly one such $S$-closure. We denote it $cl(A) = \bigcap_{A\subseteq B\leqs N} B$.
\end{definition}

\begin{definition}\label{defSimpleAmalgam}
Let $A,B_1,B_2$ be $S$-structures with $B_1\cap B_2 = A$. For each $C\in C(A)$, denote $\otn{C^1}{\mult_A(C)}\in C^{B_1(A)}$ its maximal extensions in the structure $B_1$ and $\otn{C^2}{\mult_A(C)}\in C^{B_2(A)}$ its maximal extensions in $B_2$ guaranteed by additivity (Where a maximal extension with pre-multiplicity $k$ is counted $k$ times). The order we choose for the extensions is arbitrary and fixed.

For a set $C'\subseteq B_1\cup B_2$ define:
\\$X_{C'} = \setarg{(C^1_i,C^2_i)}{C\in C(A), 1\leq i\leq \mult_A(C), C'\cap B_1= C^1_i, C'\cap B_2= C^2_i}$

Then let D be the structure with universe $B_1\cup B_2$ and $C(D)$ and $\mult_D$ defined as follows:

$C(D) = \setarg{C^1_i\cup C^2_i}{C\in C(A), 1\leq i \leq \mult_A(C)}\amalg C(B_1/A)\amalg C(B_2/A)$

$\mult_D(C) =
	\left\{
		\begin{array}{ll}
			\mult_{B_1}(C) & \text{if } C\in C(B_1/A)\\
			\mult_{B_2}(C) & \text{if } C\in C(B_2/A)\\
			|X_{C}| & \text{otherwise}
		\end{array}
	\right.$

We call the structure D a \emph{simple $S$-amalgam} of $B_1$ and $B_2$ over $A$. Ascertaining that this is an $S$-structure in accordance with our definition is left to the reader.
\end{definition}

\begin{obs}\label{simpleAmalgamObs}
\begin{enumerate}[(i)]
\item
If $D$ is a simple $S$-amalgam of $B_1$ and $B_2$ over $A$, then it is worth noting that $A,B_1,B_2$ are substructures of $D$.
\item
Let $D$ be a simple $S$-amalgam of $B_1$ and $B_2$ over $A$. Then for $A\subseteq X_1 \subseteq B_1$ and $A\subseteq X_2 \subseteq B_1$, $X_1\cup X_2$ as a substructure of $D$ is a simple $S$-amalgam of $X_1$ and $X_2$ over $A$.
\end{enumerate}
\end{obs}

\begin{remark*}
For $A,B_1,B_2$ given, there may be several simple $S$-amalgams of $B_1$ and $B_2$ over $A$. When we state that we take $D$ a simple amalgam, we will choose one arbitrarily.
\end{remark*}

\begin{lemma}\label{dosSimpleAmalgam}
Let $A,B_1,B_2$ be $S$-structures with $B_1\cap B_2 = A$. Let $D$ be a simple amalgam of $B_1$ and $B_2$ over $A$. Then $\dosover{B_2}{A} = \dosover{B_2}{B_1}$.
\end{lemma}

\begin{proof}
By definition of the simple amalgam we have $C(D/A) = C(B_1/A)\cup C(B_2/A)$ and so by Lemma \ref{dosFreeAmalgam} $\dosover{B_2}{A} = \dosover{B_2}{B_1}$.
\end{proof}

\begin{Cor}\label{leqsInSimpleAmalgam} In the notation of the above lemma: If $A\leqs B_2$ then $B_1 \leqs D$. If also $A\leqs B_1$ then $A\leqs D$.
\end{Cor}

\begin{proof}
Say $A\leqs B_2$. Let $A\subseteq X\subseteq D$ then $\dosover{X}{B_1} = \dosover{X\cap B_2}{B_1}$ then it is safe to assume $X\subseteq B_2$. By Observation \ref{simpleAmalgamObs}.i and Lemma \ref{dosSimpleAmalgam} $\dosover{X}{B_1} = \dosover{X}{A} \geq 0$. So $B_1\leqs D$.

Now say $A\leqs B_1$ also, then by transitivity of $\leqs$, $A\leqs D$.
\end{proof}

\section{The age of $M_s$}\label{age}
$\calC = \setarg{A}{A\text{ is a finite }R\text{-structure with }\emptyset \leqslant A}$ is the age of $M$. We are looking for an analogous description of the age of $M_s$.

In order to do so, we attempt to extract information regarding a structure's $R$-predimension from its $S$-predimension. We first present and discuss the notion of \emph{outsourcing} $S$-cliques and how to apply it.

\subsection{Outsourcing A Clique}\label{ageSubsecOne}
\begin{notation} In the context of $R$-structures, let $K$ be a set of cliques. Let $A$ be some $R$-structure. Denote as follows:
\begin{itemize}
\item
$W_K = \bigcup_{C\in K}wit(C)$, the set of elements witnessing cliques from $K$.
\item
$W_A = W_{C(A)}$
\item
$EW_K(A) = W_K\setminus (A\cup\bigcup K)$, the set of witnesses of cliques in $K$ that are not members of $A$ or of cliques in $K$.
\item
$EW^1_K(A) = \setarg{x\in EW_K(A)}{x\text{ witnesses exactly one clique from }K}$
\end{itemize}
\end{notation}

\begin{definition}\label{outsourceDef}
Let $A\subseteq B$ be $R$-structures. Let $K \subseteq C(B)$ such that $W_K\subseteq B$. Let a specific clique $(C,\set{x,y})\in K$ be given. Define:
\begin{enumerate}[(1)]
\item
Let $z\in \set{x,y}$. Define $z^C = z$ if $z\in EW^1_K(A)$, and otherwise let $z^C$ be some completely new element.
\item
\[\outsource{B}{A}{C} = B\cup\set{x^C,y^C}\]
\[R(\outsource{B}{A}{C}) = RW(C, \set{x^C,y^C}) \cup (R_A\setminus RW(C,\set{x,y}))\cup \bigcup_{D\in K\setminus\set{C}} RW(D)\]
\end{enumerate}
The set $\outsource{B}{A}{C}$ with the relation $R(\outsource{B}{A}{C})$ is \emph{the structure $B$ with core $A$ after the outsourcing of the clique $\clqwit{C}{x,y}$ with respect to $K$}.
\end{definition}

Outsourcing is a way for us to simplify a structure while leaving certain aspects of the structure unchanged, namely, the $R$-diagram of $A$ and the cliques of $K$.

Intuitively, we add a pair of witnesses external to the structure, remove the relations that generate the clique $C$, connect the new witnesses to the clique $C$ such that it will reform as a clique, and finally, we add back any relations that we removed that are used in other cliques in $K$. To be more precise, an added witness is not necessarily new, but this only happens when a witness is virtually external.

We prove that our description of $\outsource{B}{A}{C}$ is accurate.

\begin{lemma}\label{outsourcingRespectsS} Let $A\subseteq B$ be $R$-structures with $K\subseteq C(B)$ and $\clqwit{C}{x,y}\in K$. Denote $\outsource{B}{A}{C}$ the structure $B$ with core $A$ after the outsourcing of the clique $\clqwit{C}{x,y}$ with respect to $K$, where the new witnesses are denoted $x^C,y^C$. Then:
\begin{enumerate}[(i)]
\item
Let $\clqwit{D}{a,b}\in C(B)\setminus\set{\clqwit{C}{x,y}}$. Then $\witclq{a,b}{\outsource{B}{A}{C}} \subseteq \witclq{a,b}{B} = D$. (i.e. the clique $\clqwit{D}{a,b}$ gained no new elements during the outsourcing)
\item
$K\setminus\set{\clqwit{C}{x,y}}\subseteq C(\outsource{B}{A}{C})$
\item
Let $\set{a,b}$ witness a clique in $\outsource{B}{A}{C}$. Then $\set{a,b}$ already witness a clique in $B$ unless $\set{a,b} = \set{x^C,y^C}$. This means that no completely new cliques were unintentionally formed during the outsourcing.
\item
$RW(C,\set{x^C,y^C})\cap RW(D) = \emptyset$ for all $D\in C(\outsource{B}{A}{C})\setminus\set{(C,\set{x^C,y^C})}$.
\end{enumerate}
\end{lemma}

\begin{proof}
\begin{enumerate}[(i)]
\item
Say $(c,a,b)\in R(\outsource{B}{A}{C})$. If $(c,a,b)\notin R_B$ then because all relations added to the structure contain both $x^C$ and $y^C$, we must have $x^C,y^C\in \set{c,a,b}$ and therefore $\set{x^C,y^C}\cap \set{a,b} \neq \emptyset$. By assumption, $x^C,y^C$ do not witness a clique that is not $\clqwit{C}{x,y}$ and so this is impossible. Thus $(c,a,b)\in R_B$ and so $c\in D$.

\item
Say $\clqwit{D}{a,b}\in K\setminus\set{\clqwit{C}{x,y}}$. None of the generating relations of $D$ was removed from the structure because $\clqwit{D}{a,b}$ is in $K\setminus\set{\clqwit{C}{x,y}}$. Therefore $\clqwit{D}{a,b}$ did not lose or gain (by $(i)$) any elements during the outsourcing and $\clqwit{D}{a,b} \in C(\outsource{B}{A}{C})$.

\item
Let $\set{a,b}$ witness a clique in $\outsource{B}{A}{C}$. Let $D\subseteq \outsource{B}{A}{C}$ be the universe of a clique in $\outsource{B}{A}{C}$ witnessed by $\set{a,b}$. If $RW(D,\set{a,b})\subseteq R_B$ then we are done because $\set{a,b}$ witness a clique in $B$, so assume this is not the case. Let $(c,a,b)\in RW(D,\set{a,b})\setminus R_B$, so $\set{x^C,y^C}\subseteq \set{c,a,b}$ and so without loss of generality we have $b = y^C$. By definition of $y^C$ and $R(\outsource{B}{A}{C})$, the only relations in $R(\outsource{B}{A}{C})$ containing $y^C$ are in $RW{\clqwit{C}{x^C,y^C}}$. Because $b = y^C$ appears in all generating relations of $D$, it must be that both $x^C$ and $y^C$ appear in all relations in $RW(D,\set{a,b})$. So we must also have $a=x^C$. Thus either $\set{a,b}$ witness a clique in $A$ or $\set{a,b} = \set{x^C,y^C}$ and we are done.

\item
Let $\clqwit{D}{a,b}\in C(\outsource{B}{A}{C})\setminus\set{\clqwit{C}{x^C,y^C}}$. We have shown that no completely new cliques were formed and no cliques gained new elements so $RW{\clqwit{D}{a,b}}\subseteq R_B$. If $\set{x^C,y^C}\nsubseteq B$ then $RW{\clqwit{C}{x^C,y^C}}\cap R_B = \emptyset$ and we are done. So assume $\set{x^C,y^C}\subseteq B$. By definition of $x^C$ and $y^C$ this must mean $\set{x,y} = \set{x^C,y^C}$ and $x,y\in EW^1_K(A)$. So the only generating relations $x$ or $y$ appear in, are relations in $RW{\clqwit{C}{x,y}}$. Since in this case $\set{a,b}\cap \set{x,y} = \emptyset$ we also have $RW{\clqwit{D}{a,b}}\cap RW{\clqwit{C}{x,y}} = \emptyset$.
\end{enumerate}
\end{proof}

\begin{obs}\label{outsourceObs}\ 
Notation is as in the above lemma.
Denote:
\[\alpha = \set{x,y}\cap EW^1_K(A)\]
\[\beta = RW{\clqwit{C}{x,y}}\cap \bigcup_{D\in K\setminus \set{C}} RW(D)\]

Observe that:
\[\outsource{B}{A}{C} = B \amalg (\set{x^C,y^C}\setminus\alpha)\]
\[R(\outsource{B}{A}{C}) \subseteq (R_B\setminus RW{\clqwit{C}{x,y}})\amalg \beta \cup RW{\clqwit{C}{x^C,y^C}}\]

Then:
\begin{align*}
d_0(\outsource{B}{A}{C}) &= |\outsource{B}{A}{C}| - |R(\outsource{B}{A}{C})|
\\&\geq (|B| + (2-|\alpha|)) - ((|R_B| - |RW{\clqwit{C}{a,b}}|) + |RW{\clqwit{C}{x^C,y^C}}| + |\beta|)
\\&= (|B| + 2 - |\alpha|) - (|R_B| - |C| + |C| + |\beta|)
\\&= |B| - |R_B| + 2 - |\alpha| - |\beta|
\\&= d_0(B) + 2 -(|\alpha| + |\beta|)
\end{align*}
\end{obs}

Note that if $\alpha\neq \emptyset$ then $\beta = \emptyset$: Let $z\in \alpha$. Let $(a,b,z) RW{\clqwit{C}{x,y}}$. By definition, $z$ is not a member of any cliques in $K$ and $z$ witnesses exactly one clique in $K$ (i.e. the clique $\clqwit{C}{x,y}$). Therefore, $(a,b,z)$ cannot be a generating relation of another clique in $K$ and thus $\beta = \emptyset$.

So $(|\alpha|+|\beta|) = max\set{|\alpha|,|\beta|}$, and therefore:

\[d_0(\outsource{B}{A}{C}) = d_0(B) + (2 - max\set{\alpha,\beta})\]

Since $|\alpha|\leq 2$ always, it follows that $d_0(\outsource{B}{A}{C})\geq d_0(A)$ unless $|\beta| > 2$.

The next lemma shows there is always a clique $C\in K$ that we can outsource such that $d_0(\outsource{B}{A}{C})\geq d_0(A)$.

\begin{definition} Let $\emptyset\leqslant A$ be an $\set{R}$ structure and $K = \set{\otn{C}{n}}\subseteq C(A)$ maximal $S$-cliques. We define the graph of $A$ with respect to $K$, $G^{A}_{K} = (V,E)$ with:
\\$V = \setarg{wit(C_i)}{1\leq i\leq n}$
\\$E = \setarg{(wit(C_i),wit(C_j))}{RW(C_i)\cap RW(C_j)\neq\emptyset}$
\end{definition}

Less formally: The vertices of $G^A_{K}$ are the cliques of $K$, each represented by its pair of witnesses, and there is an edge between two cliques that use the same $R$-relation.

\begin{lemma}\label{goodclique}
Let $\emptyset\leqslant A$ be an $R$-structure and $K = \set{\otn{C}{n}}\subseteq C(A)$ with $W_K\subseteq A$. Then there exists a clique $C_k$ such that $|\beta_k| = |RW(C_k)\cap \bigcup_{\substack{i<n\\i\neq k}} RW(C_i)| \leq 2$.
\end{lemma}

\begin{proof}
Consider the undirected graph $G = G^A_K = (V,E)$.

As in Observation \ref{sharingRelation} a relation used by two distinct cliques $C_i,C_j$ must contain both pairs of witnesses, and so the relation is $wit(C_i)\cup wit(C_j)$. So each edge of the graph corresponds to a single, uniquely determined relation.

In case one relation $\set{a,b,c}$ is used in three different cliques (that must be represented by witnesses $\set{a,b},\set{b,c},\set{a,c}$) we call it a \emph{strong triangle} as it creates a triangle in the graph. We call an edge that is not a part of a strong triangle a \emph{simple edge}. There may be triangles in the graph that are not strong triangles, we treat these as three simple edges.

Note that for a single $R$-relation shared among different cliques in $K$, either it corresponds to one simple edge or to three edges (a strong triangle). Also note that this implies that distinct strong triangles cannot share an egde. Finally, observe that if a vertex is a part of a strong triangle, there are two edges corresponding to the same relation originating from that vertex. Thus, the number of edges originating in a vertex is not necessarily the number of relations the clique associated with the vertex shares.

Assume without loss of generality that $G$ is connected (otherwise take a connected component and refer to $n$ as the number of vertices in the component). We may assume this because a clique shares relations only with members of its connected component (better yet, only with its neighbours in the graph), so if we find a clique that shares at most two relations with cliques in its component, then it also shares at most two relations in general.

\dropline
Observe that $|\beta_k|\leq (n-1)$ for any $k\in\set{1,...,n}$, because any two cliques share at most one relation and there are at most $(n-1)$ other cliques to share relations with. As the lemma is trivially satisfied whenever $n<4$, we assume $n\geq 4$.

Assume that a clique as described in the lemma does not exist. So each clique uses at least 3 different relations used in other cliques. That is, as there are $n$ cliques, there must be at least $3n$ instances of a relation being used multiple times (where a simple edge is counted twice and a strong triangle is counted thrice). We henceforth call such an instance a \emph{recycling} of a relation, i.e. there must be at least $3n$ recyclings in our structure.

A simple edge implies two recyclings. A strong triangle implies three recyclings. Say in $G$ there are $t$ strong triangles and $s$ simple edges. One of the following three cases must occur and we show that for any of them $\emptyset \nleqslant A$:
\dropline
\\\underline{\textbf{Case $s=0$}}:
As $G$ has no simple edges, it is composed of strong triangles.
\dropline

\noindent\textbf{Claim.} There are at most (n-1) different points from $A$ appearing as witnesses in vertices of $V$. Namely, $|\bigcup V|\leq n-1$.
\begin{proof} Since $n>3$ and $G$ is connected, $G$ has at least two strong triangles and so $n\geq 5$. We prove inductively that for every $5\leq i \leq n$ there exists a set of $i$ vertices with at most $i-1$ witnesses appearing in them.

\underline{$i = 5$}: Since $G$ is connected, composed only of strong triangles and has at least two strong triangles, there are two strong triangles with a common vertex, as in the illustration below. Let the two strong triangles be $\set{\set{a,b},\set{b,c},\set{a,c}}$ and $\set{\set{a,b},\set{b,d},\set{a,d}}$ with $\set{a,b}$ the common vertex. There must be exactly four distinct points from $A$ among the witnesses of the five vertices forming the strong triangles. Namely, $5$ vertices, $(5-1)$ witnesses.

\bigskip
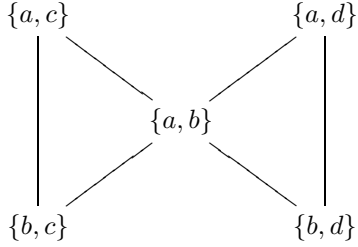
\centerline{\xymatrix{
\set{a,c} \ar@{-}[rd] \ar@{-}[dd]	&	& \set{a,d}	\ar@{-}[ld] \ar@{-}[dd]\\
	& \set{a,b}	\ar@{-}[rd] \ar@{-}[ld]	\\
\set{b,c}	&	& \set{b,d}
}}

\bigskip
\underline{$i+1\leq n$}: Assume there is a set $I$ of $i$ vertices with $i-1$ witnesses appearing in them. Since $G$ is connected, choose $v\in V\setminus I$ connected by an edge to some $u\in I$. As $(v,u)\in E$, $u$ and $v$ share a witness that is already amongst the $i-1$ witnesses appearing in $I$'s vertices. Therefore, $I\cup \set{u}$ is a set of vertices of size $i+1$ with at most $(i-1)+1 = i$ witnesses as required.
\\For $i=n$ the set of vertices must be $V$ and there are at most $n-1$ witnesses appearing in it.
\end{proof}
As there are at least $3n$ recyclings and each strong triangle contributes $3$ recyclings, there must be at least $n$ strong triangles in the graph. Each strong triangle corresponds to a specific relation among witnesses and so: $d_0(\bigcup V) = |\bigcup V| - r(\bigcup V) \leq (n-1) - n < 0$. Since $\bigcup V \subseteq A$, we get $\emptyset \nleqslant A$.
\dropline
\\\underline{\textbf{Case $t=0$}}:
As $G$ has no strong triangles, it is composed of simple edges.

We claim there are at most $(n+1)$ distinct points from $A$ appearing in $V$:

We build the set of vertices $V$ inductively, such that $V_i\subseteq V$ is a set of $i$ vertices of $V$ with at most $i+1$ points of $A$ appearing in vertices of $V_i$.

Choose an arbitrary vertex $v\in V$, denote $V_1 = \set{v}$. There are $2$ points from $A$ appearing in $v$. Given $V_i$ a set of $i<n$ vertices with at most $(i+1)$ witnesses appearing in them, choose $u\in V\setminus V_i$ connected by an edge to some vertex $v\in V_i$. Such a $u$ exists because $G$ is connected. Define $V_{i+1} = V_i\cup \set{u}$. Since $u$ was connected by an edge to an element of $V_i$, there is at most one witness appearing in $u$ not already appearing in $V_i$. Therefore, $V_{i+1}$ is a set of $(i+1)$ vertices with at most $(i+2)$ points from $A$ appearing in it. After $n$ steps we get $V_n = V$ with at most $(n+1)$ points from $A$ appearing in $V$.

As each simple edge contributes $2$ recyclings and there are at least $3n$ recyclings, there must be at least $\frac{3n}{2}$ simple edges in $G$. Each simple edge represents a relation between witnesses. Recall $n>3$ and so we have: $d_0(\bigcup V) = |\bigcup V| - r(\bigcup V) \leq (n+1) - \frac{3n}{2} < 0$. Since $\bigcup V \subseteq A$, we get $\emptyset \nleqslant A$.
\dropline
\\\underline{\textbf{Case $t>0,s>0$}}:
We claim there are at most $n$ distinct points from $A$ appearing in $V$. By assumption $t>0$ and so there is $\set{u,v,w}$ a strong triangle in $G$. Denote $V_3 = \set{u,v,w}$ which is a set of $3$ vertices with $3$ points from $A$ appearing in them. As in the previous case, we proceed inductively. Let $V_i$ be a given set of $i$ vertices with $i$ with points from $A$ appearing in them. Let $z\in V\setminus V_i$ be connected to one of the vertices of $V_i$ and define $V_{i+1} = V_i\cup \set{z}$. As before, $V_{i+1}$ has $i+1$ vertices with at most $i+1$ points from $A$ appearing in them. With each step the size of the set increases by one and the number of points from $A$ appearing in the set increases by at most one. We then have $V_n = V$ has at most $n$ points from $A$ appearing in it.

As each simple edge contributes $2$ recyclings, each strong triangle contributes $3$ recyclings and there are at least $3n$ recyclings it must be that $3n\leq 2s + 3t$. Because $s>0$ we have: $3n \leq 2s+3t<3s+3t=3(s+t)$ and so $n<s+t$. Note now that as each simple edge and each strong triangle represent a single distinct relation among witnesses, $r(\bigcup V) \geq s+t > n$. And finally:
$d_0(\bigcup V) = |\bigcup V| - r(\bigcup V) < n - n = 0$. Since $\bigcup V \subseteq A$, we get $\emptyset \nleqslant A$.
\dropline

So if $\emptyset \leqslant A$, a clique as described in the lemma must exist and so the proof is complete.
\end{proof}

\subsection{Outsourcing A Set of Cliques}\label{ageSubsecTwo}

\begin{notation}
In the context of $R$-structures, let $K$ and $L$ be sets of cliques. Let $A$ be some $R$-structure. Denote as follows:
\begin{itemize}
\item
$EW_{K,L}(A) =  EW_{K\cup L}(A)\setminus W_L$, witnesses for cliques in $K$ that are not elements of $A$, members of cliques in either $K$ or $L$ or witnesses of cliques in $L$.

\item
$EW_{K,L}^1(A) = \setarg{x\in EW_{K,L}(A)}{x\text{ witnesses exactly one clique from }K}$
\end{itemize}
\end{notation}

\begin{definition} Let $A$ be a finite $R$-structure with $\emptyset \leqslant A$. We say $L\subseteq C(A)$ and $K\subseteq C(A)$ are \emph{mutually exclusive} if there are no cliques $C_1\in L$ and $C_2\in K$ such that $RW(C_1)\cap RW(C_2) \neq \emptyset$.
\end{definition}

\begin{definition}
Let $A\subseteq B$ be $R$-structures and let $K,L\subseteq C(B)$. Let $\calW_K = \setarg{x^C,y^C}{C\in K}$ be a set of $2|K|$ distinct new elements we call such elements \emph{designated witnesses}. Define:

\[\outsource{B}{A}{K,L} = (B\setminus EW_{K,L}(A))\amalg \calW_K\]
\[R(\outsource{B}{A}{K,L}) = \bigcup_{C\in L}RW{\clq{C}}\cup(R_A\setminus \bigcup_{C\in K}RW(C,wit(C))\amalg \bigcup_{C\in K} RW(C,\set{x^C,y^C})\]

The set $\outsource{B}{A}{K,L}$ with the relation $R(\outsource{B}{A}{K,L})$ is \emph{the structure $B$ with core $A$ after outsourcing $K$ with respect to $L$}.
\end{definition}

We now define an iterative process which alternatively defines the outsourcing of a set of cliques. The process allows us to keep track of the resulting structure's pre-dimension.

\begin{definition}
We say a tuple $(B,A,I,T)$ is \emph{valid} if $A\subseteq B$ are $R$-structures and $I\subseteq T\subseteq C(B)$ with $W_T\subseteq B$. We say the tuple is \emph{good} if in addition, $I$ and $T\setminus I$ are mutually exclusive.
\end{definition}

\begin{definition}\label{A^Kdefinition}
Let $A\subseteq B$ be finite $R$-structures with $\emptyset \leqslant B$. Let $K,L \subseteq C(B)$ mutually exclusive with $W_{K\cup L}\subseteq B$. We define a new structure inductively.

Define $(B_0,A_0,I_0,T_0) = (B,A,K,K\cup L)$.
Let $(B_i,A_i,I_i,T_i)$ be given and valid with $I_i \neq \emptyset$. Let $\clq{C}\in I_i$ be a clique as guaranteed by Lemma \ref{goodclique} for the $R$-structure $B_i$ and the set of cliques $I_i$. Define $B_{i+1} =\outsource{B_i}{A_i}{C}$, the structure $B_i$ with core $A_i$ after outsourcing $\clq{C}$ with respect to $T_i$ using external witnesses $\set{x^C,y^C}$. Define $A_{i+1} = A_i$. Let $I_{i+1} = I_i\setminus\set{\clq{C}}$ and let $T_{i+1} = (T_i\setminus\set{\clq{C}})\cup\set{\clqwit{C}{x^C,y^C}}$. Obviously $A_{i+1}\subseteq B_{i+1}$, $I_{i+1}\subseteq T_{i+1}$ and $T_{i+1} \subseteq C(A_{i+1})$, by Lemma \ref{outsourcingRespectsS}.ii, so $(B_{i+1},A_{i+1},I_{i+1},K_{i+1})$ is valid.

The process ends after $|K|$ steps, because $|I_{i+1}| = |I_i| - 1$ and $|I_0| = |K|$. We say $B_{|K|}$ is a \emph{gradual outsourcing of $K$ with core $(A,L)$}.
\end{definition}

\begin{obs}
It may seem that the above process defines a family of structures, because the resulting structure depends on the order in which cliques are outsourced. In fact, all choices lead to isomorphic structures, and it is not hard to believe that, when $K$ is not empty, any such gradual outsourcing is isomorphic to $\outsource{B}{A}{K,L}$. The proof is technical and somewhat tedious and therefore may be found in the appendix (Lemma \ref{gradualIsOutsourcingLemma}).
\end{obs}

\begin{lemma} \label{outsourcePredimension}
Let $A\subseteq B$ be finite $R$-structures with $\emptyset \leqslant B$. Let $K,L \subseteq C(B)$ be mutually exclusive with $W_{K\cup L}\subseteq B$. Then $d_0(B)\leq d_0(\outsource{B}{A}{K,L})$.
\end{lemma}

\begin{proof} If $K = \emptyset$ then this is trivial. Assume $K$ is not empty.

Let $B_{|K|}$ be a gradual outsourcing of $K$ with core $(A,L)$. As in the above observation $d_0(B_{|K|}) = d_0(\outsource{B}{A}{K,L})$. We analyse the process by which $B_{|K|}$ was defined.

The tuple $(B_0,A_0,I_0,T_0)$ is good because by assumption the cliques $I_0 = K$ and $T_0\setminus I_0 \subseteq L$ are mutually exclusive.

We prove that if $(B_{i},A_{i},I_{i},T_{i})$ is good then $(B_{i+1},A_{i+1},I_{i+1},T_{i+1})$ is good. Say in this stage the outsourced clique was $C$. Recall that, as discussed in Lemma \ref{outsourcingRespectsS}, the outsourcing of $C$ does not change any of the other cliques in $T_i$. Let $D_1\in T_{i+1}\setminus I_{i+1}$ be a maximal clique in $B_{i+1}$ and let $D_2$ be any maximal clique in $I_{i+1}$.

If both $D_1\neq \clqwit{C}{x^C,y^C}$ then both $D_1$ and $D_2$ are unchanged from the previous stage, meaning $D_1\in T_i\setminus I_i$ and $D_2\in T_i$, therefore $RW(D_1) \cap RW(D_2) = \emptyset$ by goodness of $(B_i,A_i,I_i,T_i)$.

If $D_1 = C$ then by Lemma \ref{outsourcingRespectsS}.iv $D_1$ shares no relation with any other clique in $B_{i+1}$, in particular $D_2$. So again $RW(D_1) \cap RW(D_2) = \emptyset$.

In any case, $RW(D_1) \cap RW(D_2) = \emptyset$. Thus, $(B_{i+1},A_{i+1},I_{i+1},T_{i+1})$ is good.

We show monotonicity of predimension. For a given stage $i$, let $C\in I_i$ be the clique to be outsourced. By goodness of the tuple $(B_i,A_i,I_i,K_i)$, we have $RW(C) \cap RW(D) = \emptyset$ for any $D\in T_i\setminus I_i$. And so 
\[RW(C)\cap \bigcup_{D\in T_i\setminus\set{C}} RW(D) = RW(C)\cap \bigcup_{D\in I_i\setminus\set{C}} RW(D)\]
By assumption, $C$ satisfies the conclusion of Lemma \ref{goodclique}, and therefore
\[|RW(C)\cap \bigcup_{D\in I_i\setminus\set{C}} RW(D)| \leq 2\]

By Observation \ref{outsourceObs} we have $d_0(B_{i}) \leq d_0(B_{i+1})$. By transitivity, for $i\leq j$ we get $d_0(B_i) \leq d_0(B_j)$ and specifically $d_0(A) = d_0(B_0) \leq d_0(B_{|K|}) = d_0(\outsource{B}{A}{K,L})$ as desired.
\end{proof}

\subsection{$A\leqslant M \implies A \leqs M_s$}\label{ageSubsecThree}

\begin{definition}
Let $A\subseteq B\subseteq E$ be $R$-structures.

Recall that $C(B/A)$ is the set of cliques $C\in C(B)$ which do not extend a clique from $C(A)$ (i.e. $|C\cap A| \leq 2$).

We define $C^E(B/A)$ to be the set of cliques in $E$, extending a clique from $C(B/A)$. Namely $C^E(B/A) = \setarg{C\in C(E)}{C\cap B\in C(B/A)}$.

An equivalent definition is $C^E(B/A) = C(E/A)\setminus C(E/B)$.
\end{definition}

\begin{lemma}\label{outsourceDoAndDos}
Let $A\subseteq B\subseteq N$ be $R$-structures. Denote $E = B\cup W_B$, $K = C^E(B/A)$, $L = C^{E}(A)$. Assume that $W_L\subseteq A$ and $RW(C)\cap R_A = \emptyset$ for any $C\in K$. Then $d_0(\outsource{E}{A}{K,L}/A)\leq \dosover{B}{A}$. If in addition $E\setminus EW_{K,L}(A) = B$, equality holds.
\end{lemma}

\begin{proof}
Since $W_L\subseteq A$ we have $EW_{K,L}(A) = EW_K(A)$. Also, by assumption, $R_A\setminus \bigcup_{C\in K}RW(C,wit(C)) = R_A$. So in this case

\[\outsource{E}{A}{K,L} = (E\setminus EW_{K}(A))\amalg \calW_K\]
\[R(\outsource{E}{A}{K,L}) = R_A\cup \bigcup_{C\in L}RW{\clq{C}}\amalg \bigcup_{C\in K} RW(C,\set{x^C,y^C})\]

Let $D$ be the induced $R$-structure on the set $(B\setminus EW_K(A))\amalg \calW_K \subseteq\outsource{E}{A}{K,L}$.

\medskip
\noindent\textbf{Claim 1.} $d_0(\outsource{E}{A}{K,L}/A)\leq d_0(D/A)$. If in addition $E\setminus EW_K(A) = B$, equality holds.

\begin{proof} Let $w\in (E\setminus B)\setminus EW_K(A)$. Then $w\in W(B)$ and $w\in C$ for some $C\in K$, meaning $(w,x^C,y^C)\in R(\outsource{E}{A}{K,L})$. This means that there exists an injection between elements of $(E\setminus B)\setminus EW_K(A)$ and relations in $R(\outsource{E}{A}{K,L})$ and therefore $d_0(\outsource{E}{A}{K,L})\leq d_0(D)$. Since $A$ is a substructure of both we have $d_0(\outsource{E}{A}{K,L}/A)\leq d_0(D/A)$.

If in addition $E\setminus EW_K(A) = B$ then $D = \outsource{E}{A}{K,L}$ and equality holds trivially.
\end{proof}

We wish to write $D$ explicitly. We observe several facts.

\begin{itemize}
\item
The map $\witclq{x,y}{E}\mapsto\witclq{x,y}{B}$ induces bijections between $K \rightarrow C(B/A)$ and $L\rightarrow C^B(A)$. Both the bijections fix pairs of witnesses and also $\witclq{x,y}{E}\supseteq\witclq{x,y}{B}$. Thus $W_{C^B(A)} = W_L$ and $W_{C(B/A)} = W_K$. By definition, the second equality yields $EW_K(A)\subseteq EW_{C(B/A)}(A)$.

In addition, let $w\in EW_{C(B/A)}(A)\setminus EW_K(A)$, then $w\in \witclq{x,y}{E}$ for some $\set{x,y}$ witnessing a clique in $C(B/A)$, but it must be that $w\notin \witclq{x,y}{B}$ and so $w\notin B$.

Combining the two facts gives us $B\setminus EW_K(A) = B\setminus EW_{C(B/A)}(A)$.
\item
By the bijective map above we may redefine $\calW_K = \setarg{x^C,y^C}{C\in C(B/A)}$.

\item
Let $\clqwit{C}{x,y}\in L$ and let $(a,x,y)\in RW{\clqwit{C}{x,y}}$. Since $W_L\subseteq A\subseteq D$ we have $(a,x,y)\in R_D \iff a\in D \iff a\in B\setminus EW_K(A)\iff a\in B$. Thus $RW{\clqwit{C}{x,y}}\cap R_D = RW{\clqwit{C\cap B}{x,y}}$. The same reasoning also gives $RW{\clqwit{C}{x^C,y^C}}\cap R_D = RW{\clqwit{C\cap B}{x^C,y^C}}$ for any $C\in K$.

Recall as in the bijective map above that $C\cap B\in C^B(A)$ for $C\in L$ and $C\cap B\in C(B/A)$ for $C\in K$.

\item
Since $A\in D$ we have $R_D\subseteq R_D$
\end{itemize}

We can now write $D$ explicitly

\[D = (B\setminus EW_{C(B/A)}(A))\amalg \setarg{x^C,y^C}{C\in C(B/A)}\]
\[R_D = R_A\cup \bigcup_{C\in C^{B}(A)} RW(C, wit(C))\cup \bigcup_{C\in C(B/A)}RW(C,\set{x^C,y^C})\]

\medskip
\noindent\textbf{Claim 2.} $d_0(D/A)\leq \dosover{B}{A}$. If in addition $E\setminus EW_K(A) = B$, equality holds.

\begin{proof} Recall that for $A$ an $R$-structure, $i(A) = \sum_{C\in C(A)} |C|$. We explicitly calculate the terms:
\begin{align*}
d_0(D/A) &= |(B\setminus EW_{C(B/A)}(B))| + |\setarg{x^C,y^C}{C\in C(B/A)}| - |A| - (|R_D| - |R_A|)
\\&\leq |B|-|A| + 2|C(B/A)| - (|R_D\setminus R_A|)
\\\dosover{B}{A} &= |B| - |A| + 2|C(B/A)| - (i(B)-i(A))
\end{align*}

Now, in order to prove the claim, it is enough to show that $i(B) - i(A) = |R_D\setminus R_A|$.

\begin{remark*}
Note that the first inequality is an equality in case $E\setminus EW_K(A) = B$. In that case, proving $i(B) - i(A) = |R_D\setminus R_A|$ gives us $d_0(D/A) = \dosover{B}{A}$.
\end{remark*}

Denote \[X = \bigcup_{C\in C(B/A)}C\times\set{C}\amalg \bigcup_{C\in C^B(A)}(C\setminus A)\times \set{C}\]
Then:
\begin{align*}
i(B) - i(A) &= \sum_{C\in C(B)} |C| - \sum_{C\in C(A)} |C|
\\&= \sum_{C\in C(B/A)} |C| + \sum_{C\in C^B(A)} |C\setminus A| = |X|
\end{align*}
\medskip
\\Consider the following map:
\\$f:X \longrightarrow R_D$
\\$f(c,C) =
\left\{
	\begin{array}{ll}
		\set{c}\cup wit(C)  & \mbox{if } C\in C^B(A) \\
		\set{c,x^C,y^C} & \mbox{if } C\in C(B/A)
	\end{array}
\right.$

Note that $f(c,C)$ is a generating relation of $C$ in the structure $D$.We show that $f$ is a bijection from $X$ onto $R_D\setminus R_A$

\begin{itemize}
\item
We show that $f$ is into $R_D\setminus R_A$.

Let $(c,C)\in X$.
\\If $C\in C^B(A)$ then $c\notin A$ and so $f(c,C)\notin R_A$.
\\If $C\in C(B/A)$ and $\set{x^C,y^C}\nsubseteq A$ then $f(c,C)\notin R_A$.
\\If $C\in C(B/A)$ and $\set{x^C,y^C}\subseteq A$ then $\set{x^C,y^C} = wit(C)$ and $\witclq{x^C,y^C}{B}\in C(B/A)$. We know that $\witclq{x^C,y^C}{E}\in K$ and $\witclq{x^C,y^C}{B}\subseteq \witclq{x^C,y^C}{E}$. So by assumption $RW{\clqwit{C}{x^C,y^C}}\cap R_A = \emptyset$. Since $f(c,C)\in RW{\clqwit{C}{x^C,y^C}}$ it must be that $f(c,C)\notin R_A$.

\item
We show that $f$ is injective.

Assume $f(c_1,C_1) = f(c_2,C_2)$. If $C_1 = C_2$ then obviously $c_1=c_2$, so we assume $C_1\neq C_2$. So $C_1$ and $C_2$ share the generating relation $c_1\cup wit(C_1) = c_2\cup wit(C_2)$. A clique witnessed by designated witnesses shares no relation with any other clique and thus it must be that $C_1,C_2\in C^B(A)$. So $c_1\in wit(C_2)$ and by assumption $wit(C_2)\subseteq W_L \subseteq A$. But we know $c_1\in C_1\setminus A$, a contradiction. So it must be that $(c_1,C_1) = (c_2,C_2)$ and thus $f$ is an injection.

\item
We show that $f$ is surjective.

Let $r\in R_D\setminus R_A$.
\\If $r\in RW{\clq{C}{x^C,y^C}}$ for some $C\in C(B/A)$ then it is of the form $(b,x^C,y^C)$, thus $(b,C)\in X$ and $f(b,C) = r$.
\\If $r\in RW{\clqwit{C}{x,y}}$ for some $C\in C^B(A)$, say $r = (b,x,y)$, then $\set{x,y}\subseteq W_{C^B(A)} = W_L$ and so by assumption $\set{x,y}\subseteq A$. Since $r\notin R_A$ it must be that $b\notin A$ and thus $b\in C\setminus A$. So $(b,C)\in X$ and $f(b,C) = r$.
\end{itemize}
\medskip
$f$ is a bijection from $X$ onto $R_D\setminus R_A$ and so $|X|= |R_D\setminus R_A|$. Then $i(B) - i(A) = |X| = |R_D\setminus R_A|$ and consequently $d_0(D/A) \leq \dosover{B}{A}$.

By the above remark, if in addition $E\setminus EW_K(A) = B$, we have $d_0(D/A)= \dosover{B}{A}$.
\end{proof}

By combining  Claim $1$ and Claim $2$ we have $d_0(\outsource{E}{A}{K,L}/A)\leq d_0(D/A) \leq \dosover{B}{A}$. In addition, if $E\setminus EW_K(A) = B$, both inequalities are equalities and equality holds.
\end{proof}

\begin{Cor}\label{genericBoverA}
Let $A\subseteq B\subseteq M$. Denote $E = B\cup W_B$, $K = C^E(B/A)$, $L = C^{E}(A)$. Assume that $K$ and $L$ are mutually exclusive, $W_L\subseteq A$ and $RW(C)\cap R_A = \emptyset$ for any $C\in K$. Then $d_0(E/A)\leq \dosover{B}{A}$.
\end{Cor}

\begin{proof}
All requirements of Lemma \ref{outsourcePredimension} are fulfilled and so $d_0(E)\leq d_0(\outsource{E}{A}{K,L})$. Since $A$ is a substructure of both structures, we also have $d_0(E/A)\leq d_0(\outsource{E}{A}{K,L}/A)$. By Lemma \ref{outsourceDoAndDos} we have $d_0(\outsource{E}{A}{K,L}/A) \leq \dosover{B}{A}$ and so $d_0(E/A)\leq \dosover{B}{A}$.
\end{proof}

\begin{lemma}\label{twoAlreadyInStruct}
Let $A\subseteq N$ be $R$-structures. Let $C$ be the universe of a clique witnessed by $wit(C) = \set{x,y}$ with $|(C\cup wit(C))\cap A|\geq 2$. Then for any $D\subseteq C$, $d_0(wit(C)\cup D/A) \leq 0$. If also $A\leqslant N$, then $(A\cup wit(C)\cup D) \leqslant N$.
\end{lemma}

\begin{proof}
We claim that $d_0(wit(C)/A) \leq 0$. We examine the three possible (trivial) cases and keep in mind that $|(C\cup wit(C))\cap A|\geq 2$:
\\\textbf{Case I $-\mathbf{|wit(C)\cap A| = 2}$:} Then $wit(C)\subseteq A$ and thus $d_0(wit(C)/A) = d_0(A/A) = 0$.
\\\textbf{Case II $-\mathbf{|wit(C)\cap A| = 1}$:} Say $wit(C)\cap A = \set{y}$. We must have $|A\cap C|\geq 1$ and so there exists $a\in C\cap A$. Then $d_0(wit(C)/A) \leq |\set{x}| - |\set{(a,x,y)}| = 0$
\\\textbf{Case III $-\mathbf{|wit(C)\cap A| = 0}$:} We must have $|A\cap C|\geq 2$ and so there exist $a_1,a_2\in C\cap A$ distinct. Then $d_0(wit(C)/A) \leq |\set{x,y}| - |\set{(a_1,x,y), (a_2,x,y)}| = 0$

So $d_0(wit(C)/A) \leq 0$. Obviously, for any $D\subseteq C$, $d_0(D/wit(C)) \leq 0$ and so $d_0(D/A\cup wit(C)) \leq 0$. Combining the two inequalities we have $d_0(A\cup wit(C)\cup D) = d_0(D/A\cup wit(C)) + d_0(wit(C)/A) + d_0(A) \leq d_0(A)$

\smallskip
If in addition $A\leqslant N$, we know $d_0(wit(C)\cup D/A) \geq 0$ and so we have $d_0(A\cup wit(C)\cup D) = d_0(A) = d(A,N)$. So $(A\cup wit(C)\cup C') \leqslant N$ .
\end{proof}

\begin{remark*}
For $A\subseteq N$ and $b\in N\setminus A$, $\dosover{b}{A} = 1 - Clq(b,A\cup\set{b})$.
\end{remark*}

\begin{prop}\label{StrongImpliesLeqs}
$A\leqslant M \implies A\leqs M_s$
\end{prop}

\begin{proof}
Assume the statement is false. Let $A\leqslant M$ and $A\subseteq B\subseteq M$ be such that $\dosover{B}{A} < 0$ with $n = |B\setminus A|$ minimal.

\noindent\textbf{\underline{Case $\mathbf{n=1}$:}} So $B = A\cup\set{b}$ for some $b\in M\setminus A$. Say $C\in C(B)$ such that $b\in C$, so it must be that $|C\cap A| \geq 2$ and so $d_0(wit(C)/A) \leq 0$.

Denote $CLQ(b,B) = \setarg{C\in C(B)}{b\in C}$. Let 
\[\bar{A} = A\cup \bigcup_{C\in CLQ(b,B)}wit(C)\]
Then $d_0(\bar{A}) \leq d_0(A)$ and since $A\leqslant M$ we have equality and $\bar{A}\leqslant M$. Since $d_0(b/\bar{A}) \leq (1-Clq(b,B))$ and $\bar{A}\leqslant M$ it must be that $Clq(b,B) \leq 1$. Thus we have $\dosover{B}{A} = 1 - Clq(b,B) \geq 0$ in contradiction to the assumption on $B$.

So for any $A\leqslant M$ and $b\in M\setminus A$ it must be that $\dosover{b}{A} \geq 0$.

\smallskip
\noindent\textbf{\underline{Case $\mathbf{n>1}$:}} Consider $E= B\cup W_B$.  Denote $K = C^E(B/A)$ and $L = C^{E}(A)$. Note that because $A\leqslant M$, we have $W_L\subseteq A$.

\smallskip
\noindent\textbf{Claim 1.} Let $C\in L$. If $b\in C\setminus A$ then $b\in W_B\setminus B$.
\begin{proof}
Assume $b\in C\setminus A$.

As explained, $wit(C)\subseteq A$, so $A\cup\set{b}\leqslant M$ and by case $n=1$, $\dosover{b}{A} \geq 0$. So if $b\in B$, the pair $A\cup\set{b} \subseteq B$ contradicts the minimality of $n$. Thus $b\in W_B\setminus B$.
\end{proof}

\smallskip
\noindent\textbf{Claim 2.} Let $C\in K$, then $RW(C)\cap R_A = \emptyset$
\begin{proof}
Assume not. Then $wit(C)\subseteq A$. Since $C\in K$ There must be an element $b\in (B\setminus A)\cap C$. So $A\cup\set{b} \leqslant M$ and by case $n=1$ we know $\dosover{b}{A} \geq 0$. The pair $A\cup\set{b} \subseteq B$ contradicts the minimality of $n$ and so our assumption is impossible.
\end{proof}

\noindent\textbf{Claim 3.} $K$ and $L$ are mutually exclusive.
\begin{proof}
Let $(C_0,\set{x_0,y_0})\in L$ and $(C_1,\set{x_1,y_1})\in K$. Again, note that $\set{x_0,y_0}\subseteq A$. Assume $RW(C_0)\cap RW(C_1) \neq \emptyset$, then $|\set{x_0,y_0}\cap \set{x_1,y_1}| = 1$. Without loss of generality the common witness is $x = x_0 = x_1$, and so $y_1\in C_{0}$. So $RW(C_0)\cap RW(C_1) = \set{(y_0,x,y_1)}$. By $\text{Claim } 2$ $(y_0,x,y_1)\notin R_A$, but since $y_0,x\in A$ it must be that $y_1\notin A$. So $y_1\in C_{0}\setminus A$ and by Claim 1, $y_1\in W_B\setminus B$.

Since $wit(C_0)\subseteq A$ and $y_1\in C_0$ we get $d_0(y_1/A) = 0$, implying $A\cup\set{y_1} \leqslant M$. By the case $n=1$ and the fact $A\leqslant M$, $\dosover{y_1}{A} \geq 0$. Since $C_1\in K$, there must be an element $b\in (B\setminus A)\cap C_1$. Again by the case $n=1$, $\dosover{b}{A,y_1} \geq 0$. We know $d_0(b/A,y_1) = 0$ and so $A\cup\set{b,y_1}\leqslant M$.

As $y_1\notin B$, by Lemma \ref{dosFreeAmalgam} we have $0 > \dosover{B}{A} \geq \dosover{B}{A,y_1} = \dosover{B}{A,b,y_1} + \dosover{b}{A,y_1} \geq \dosover{B}{A,b,y_1}$ so $\dosover{B}{A,b,y_1} <0$. On the other hand we already have $A\cup\set{b,y_1}\leqslant M$. But $0 < |B\setminus (A\cup\set{b,y_1})| = |B\setminus A\cup \set{b}| < |B\setminus A|$ in contradiction to the minimality of $n$. So there cannot exist such $C_0$, $C_1$.
\end{proof}

Claim 2 and Claim 3 assert that the conditions of Lemma \ref{genericBoverA} are met and so $d_0(E/A) \leq \dosover{B}{A} < 0$ in contradiction to $A\leqslant M$.

In conclusion, the initial assumption leads to a contradiction in any case and thus $A\leqs M_s$ as expected.
\end{proof}

\begin{Cor}\label{esiMs} $\emptyset \leqs M_s$
\end{Cor}

\subsection{Definiton of $\mathit{\calC_s}$}\label{ageSubsecFour}

The previous sub-section's conclusion assures us that $\emptyset \leqs A$ for any $A$ finite substructure of $M_s$. We wish to show that this property defines the age of $M_s$.

\begin{definition}\label{genericRepresentation} Let $\mathcal{A}$ be an $S$-structure with universe $A$.
\\Define $\mathcal{W_A} = \setarg{x^{C_i},y^{C_i}}{C\in C(A), 1\leq i\leq \mult_\mathcal{A}(C)}$
\\Define the \emph{generic $R$-representation} of $\mathcal{A}$ to be the $R$-structure $A^*$ with
\[A^* = A\amalg \mathcal{W_A}\]
\[R_{A^*} = \bigcup_{C\in C(A)}(\bigcup_{i=1}^{\mult_{\mathcal{A}}(C)} RW(C,\set{x^{C_i},y^{C_i}}))\]
Note that the $S$-structure induced on $A\subseteq A^*$ is exactly $\mathcal{A}$.
\end{definition}

\begin{lemma}\label{GenericIsEmbeddable} Let $A$ be an $S$-structure with $\esi A$. Then $\emptyset \leqslant A^*$ and $A\leqs A^*$.
\end{lemma}

\begin{proof}
Let $B^*\subseteq A^*$. Denote $B = B^*\cap A$, and denote by $\mathcal{B}$ the $S$-structure associated with $B$. Note that cliques in $B^*$ in fact contain only elements from $B$ and so $C(B^*) = C(B)$. We wish to show $d_0(B^*)\geq 0$.

Without loss of generality assume $\set{x^C,y^C}\cap B\neq\emptyset$ iff $\set{x^{C},y^{C}} \subseteq B^*$ iff $(C\cap B^*)\in C(B^*) = C(B)$. This is possible to assume since $d_0(x^C,y^C/B^*)\leq 0$ for cliques $C$ with $|C\cap B| \geq 3$, and $d_0(x^D,y^D/B^*\setminus\set{x^D,y^D})\geq 0$ for cliques $D$ with $|D\cap B| < 3$.

Now note that all R-relations within $B^*$ arise from designated witnesses' relation to their respective cliques, and no relation is used by two different cliques. Therefore
\[r(B^*) = \sum_{C\in C(B^*)} |C| = \sum_{C\in C(B)} |C|\]
\[|B^*| = |B| + 2|C(B^*)| = |B| + 2|C(B)|\]
\\And so
\[d_0(B^*) = |B^*| - r(B^*) = |B| + 2|C(B)| - \sum_{C\in C(B)}|C| = \dos(B)\]
By the assumption on $A$ we have $\dos(B) \geq 0$ and thus $d_0(B^*) \geq 0$. Since $B^*$ was general, this gives $\emptyset \leqslant A^*$.

The fact $A\leqs A^*$ is evident because the elements of $A^*\setminus A$ are not members in cliques.
\end{proof}

We denote the age of $M_s$, the set of finite substructures of $M_s$, by $\calC_s$. We conclude this section with an explicit characterization of the elements of $\calC_s$.

\begin{theorem}\label{Criterion}
$\calC_s = \setarg{A}{A \text{ is a finite }S\text{-structure with }\emptyset \leqslant_s A}$. In addition, any $S$-structure $A$ with $\emptyset \leqs A$ is strongly embeddable into $M_s$.
\end{theorem}

\begin{proof}
\underline{$\subseteq$:} Let $A$ be a finite $S$-structure with $\emptyset \leqslant_s A$. Then by Lemma \ref{GenericIsEmbeddable}, $A^*$ is strongly embeddable into $M$, we identify $A^*$ with its image in $M$. Because the embedding is strong, $C(A,A^*) = C(A,M)$. So $A$ embeds into $M$ with the same $S$-structure.

Moreover, by Proposition \ref{StrongImpliesLeqs} we have $A^*\leqs M$, and so by transitivity and the second assertion of Lemma \ref{GenericIsEmbeddable} we have $A\leqs M$. In conclusion, $A\in \calC_s$ and $A$ is strongly embeddable into $M_s$.

\underline{$\supseteq$:} Let $A\in \calC_s$ meaning $A\subset M_s$. Since $\esi M_s$, by definition $\esi A$.
\end{proof}

\section{$M_s$ as a fra\"{i}ss\'e Limit}\label{MsFraisse}

Let $\fC_s$ be the category with class of objects $\calC_s$ and morphisms being strong embeddings. We wish to show that $\fC_s$ has a countable fra\"{i}ss\'e limit and that this limit is in fact $M_s$.

\begin{prop}\label{FraisseProperties} $\fC_s$ has the following properties:
\begin{enumerate}
\item Hereditary Property - If $A\subseteq B\in \fC_s$ then $A\in \fC_s$.
\item Joint Embedding Property - for any $A,B\in \fC_s$ there exists $C\in \fC_s$ such that there are strong embeddings $i:A \rightarrow C$ and $j: B\rightarrow C$.
\item Amalgamation Property - for $A,B_1,B_2\in \fC_s$ with $A\leqslant_s B_1$ and $A\leqslant_s B_2$ there exists $D\in \fC_s$ and strong embeddings $i:B_1 \rightarrow D$ and $j: B_2\rightarrow D$ such that $i(A) = j(A)$ and $i(A)\leqs D$.
\end{enumerate} 
\end{prop}

\begin{proof} We use the fact $A\in \fC_s \iff \esi A$
\begin{enumerate}
\item Let $A\subseteq B$ and $\esi B$, So $\esi A$.
\item Let $A,B\in \fC_s$. This follows from the Amalgamation Property and amalgamating $A$ and $B$ over the empty set.
\item Let $A,B_1,B_2$ be as in the statement of the property, consider a simple amalgam of $B_1$ and $B_2$ over $A$, call it $D$. Then by Corollary \ref{leqsInSimpleAmalgam} the identities are strong embeddings of $B_1$ and $B_2$ into $D$, agreeing on $A$. Also, $A\leqs D$ and due to transitivity and $\esi A$ we have $\esi D$. So $D\in \fC_s$ is the structure sought after.
\end{enumerate}
\end{proof}

So as we desired, by Proposition \ref{FraisseProperties} there exists (by a well known back-and-forth argument) a single (up to isomorphism) countable structure with age $\fC_s$ which is homogeneous with respect to strongly embedded substructures (i.e. any finite partial isomorphism $f:A\rightarrow B$ with $A,B$ strong in the structure extends to an isomorphism of the entire structure). We call this structure the fra\"{i}ss\'e limit of $\fC_s$.

We would like to show that $M_s$ is the fra\"{i}ss\'e limit of $\fC_s$. We already know by Theorem \ref{Criterion} that $\calC_s$ is the age of $M_s$ and that there is a strong embedding of any object in $\fC_s$ into $M_s$. It is only left to show that $M_s$ is homogeneous with respect to strongly embedded substructures.

\begin{definition}\label{mixedGenericAmalgam}
Let $A\subseteq \bar{A}\subseteq N$ be $R$-structures with $W_A\subseteq \bar{A}$. Denote by $\calA$ the $S$-structure induced on $A$. Let $\calB \in \calC_s$ be some $S$-structure with universe $B$ with $\calA\subseteq\calB$. We define a family of $R$-structures.

\smallskip
Let $\fullenum{1}$ be a full-enumeration of $C^{\calB}(\calA)$. By additivity there exists a bijection $g: \fullenum{1} \rightarrow C(A)$ with $g(C) = C\cap A$, we fix such a $g$. We can now think of the cliques in $\fullenum{1}$ as witnessed and define $wit(C) = wit(g(C))$.

Let $\fullenum{2}$ be a full-enumeration of $C(\calB/\calA)$. Denote $\calW = \setarg{x^C,y^C}{C\in \fullenum{2}}$.

\smallskip
We define an $R$-structure $D$ with

\[D = \bar{A}\amalg (B\setminus A)\amalg \calW\]
\[R_D = R_{\bar{A}} \cup \bigcup_{C\in \fullenum{1}} RW{\clq{C}}\amalg \bigcup_{C\in \fullenum{2}}RW{\clqwit{C}{x^C,y^C}}\]

We call $D$ a \emph{mixed generic amalgam of $\bar{A}$ and $\calB$ over $A$}.
\end{definition}

\begin{obs}\label{mixedAmalgamProps}
Let $D$ be as in the definition above. Then:
\begin{enumerate}[i.]
\item
The $S$-diagram of $D$ is that of a simple amalgam of $\calB$ and $\bar{A}$ over $\calA$ as $S$-structures.
\item
The structure $D$ is the free amalgam of $B\cup \calW$ and $\bar{A}$ over $A\cup W_A$ as $R$-structures.
\item
The elements of $\calW$ are not members of any cliques. Therefore $C(D) = C(B\cup \bar{A})$, $C^D(B/\bar{A}) = C(B/\bar{A})$ and $C^D(\bar{A}) = C^B(\bar{A})$.
\end{enumerate}
\end{obs}

\begin{lemma}\label{mixedGenericOutsourcing}
In the notation of Definition \ref{mixedGenericAmalgam}, let $D$ be a mixed generic amalgam of $\bar{A}$ and $\calB$ over $A$. Denote $K = C^D(B/\bar{A})$ and $L = C^D(\bar{A})$. Then:
\begin{enumerate}[i.]
\item
$EW_{K,L}(\bar{A}) = \calW = \calW_K$
\item
$\outsource{D}{\bar{A}}{K,L} = D$
\end{enumerate}
\end{lemma}

\begin{proof}
\begin{enumerate}[i.]
\item
By \ref{mixedAmalgamProps}.iii we have $W_K = W_{C(B/\bar{A})}$. By the fact $D$ has the structure of a simple amalgam over $A$ we have $W_{C(B/\bar{A})} = W_{C(B/A)} = \calW$ and so $W_K = \calW$. The elements of $\calW$ are not members of cliques and $\calW$ is disjoint from $\bar{A}$, so $\calW \subseteq EW_{K\cup L}(\bar{A})$. The set $\calW$ is disjoint from $W_L$ and so $\calW \subseteq EW_{K,L}(\bar{A})$. Obviously $EW_{K,L}(\bar{A})\subseteq W_K = \calW$ and so $\calW = EW_{K,L}(\bar{A})$ as desired.

The fact $\calW = \calW_K$ is evident.
\item
Recall that
\[\outsource{D}{\bar{A}}{K,L} = (D\setminus EW_{K,L}(\bar{A}))\amalg \calW_K\]
\[R(\outsource{D}{\bar{A}}{K,L}) = \bigcup_{C\in L}RW{\clq{C}}\cup(R_{\bar{A}}\setminus \bigcup_{C\in K}RW(C,wit(C))\amalg \bigcup_{C\in K} RW(C,\set{x^C,y^C})\]

By what we showed above, $(D\setminus EW_{K,L}(\bar{A}))\amalg \calW_K = D\setminus\calW \amalg \calW = D$. Since $W_K\cap \bar{A} = \emptyset$ we have $RW{\clq{C}} \cap R_{\bar{A}} = \emptyset$ and so $R_{\bar{A}}\setminus \bigcup_{C\in K}RW(C,wit(C)) = R_{\bar{A}}$. So, rewriting, we have

\[\outsource{D}{\bar{A}}{K,L} = D\]
\[R(\outsource{D}{\bar{A}}{K,L}) = R_{\bar{A}}\cup\bigcup_{C\in L}RW{\clq{C}}\amalg \bigcup_{C\in K} RW(C,\set{x^C,y^C})\]

Which is exactly the structure $D$.
\end{enumerate}
\end{proof}

\begin{lemma}\label{mixedGenericStrong}
In the notation of Definition \ref{mixedGenericAmalgam}, let $D$ be a mixed generic amalgam of $\bar{A}$ and $\calB$ over $A$. Then:
\begin{enumerate}[i.]
\item
If $A\leqs\bar{A}$ then $B\leqs D$.
\item
If $A\leqs B$ then $\bar{A} \leqslant D$.
\end{enumerate}
\end{lemma}

\begin{proof}
\begin{enumerate}[i.]
\item
Obviously, since elements of $\calW$ are not members of cliques, we have $B\leqs B\cup \calW$. By \ref{mixedAmalgamProps}.i $D$ is a simple amalgam of $B\cup \calW$ and $\bar{A}$ over $A$. Since $A\leqs\bar{A}$, by \ref{leqsInSimpleAmalgam} we have $B\cup\calW\leqs D$. By transitivity this gives us $B\leqs D$.

\item
Let $\bar{A}\subseteq E\subseteq D$. We wish to show $d_0(E/\bar{A})\geq 0$. Without loss of generality assume $\set{x^C,y^C}\cap E \neq \emptyset \iff \set{x^C,y^C}\subseteq E \iff |C\cap E| > 2$. We may assume this since $d_0(\set{x^C,y^C}/E\setminus\set{x^C,y^C})<0$ iff $|C\cap E| > 2$. So $\calW\cap E = \setarg{x^C,y^C}{C\in C^E(B/A)}$.

Observe that under our assumption, $E$ is a mixed generic amalgam of $\bar{A}$ and $E\cap B$ (as an induced $S$-structure) over $A$. Denote $K = C^E(B/\bar{A})$ and $L = C^E(\bar{A})$. Since $A\leqs B\cap E$, these are exactly the conditions of the statement of the lemma and so we may assume without loss of generality that $E=D$.

We note the following: $\bar{A} \subseteq B\cup\bar{A}$, $E = (B\cup\bar{A})\cup W_{B\cup\bar{A}}$, $W_L\subseteq \bar{A}$, $RW(C)\cap R_{\bar{A}} = \emptyset$ for any $C\in K$ and by \ref{mixedGenericOutsourcing} also $E\setminus EW_{K,L}(\bar{A}) = B\cup\bar{A}$. These are exactly the conditions of Lemma \ref{outsourceDoAndDos} and so $d_0(\outsource{E}{\bar{A}}{K,L}/\bar{A}) = \dosover{B}{\bar{A}}$. By $A\leqs B$ and Lemma \ref{leqsInSimpleAmalgam} we have $\bar{A}\leqs D$ and so $\dosover{B}{\bar{A}}\geq 0$. By Lemma \ref{mixedAmalgamProps}.ii we have $E = \outsource{D}{\bar{A}}{K,L}$. Combining the equalities, we have $d_0(E/\bar{A}) \geq 0$. And so as explained, $\bar{A}\leqslant D$.
\end{enumerate}
\end{proof}

\begin{prop}\label{weakHomogeneity} Let $A\leqs M_s$ and $A\leqs B$ for some $B\in\calC_s$, then there is a strong embedding $f_0:B\rightarrow M_s$ with $f\upharpoonright A = Id_A$.
\end{prop}

\begin{proof}
Since $A$ is a specific subset of $M_s$, we may treat it as an $R$-substructure of $M$. Denote the $R$-self-sufficient-closure of $A$ by $\bar{A}$. Take $D$ to be a mixed generic amalgam of $\bar{A}$ and $B$ over $A$. Since $A\leqs M$, we have $A\leqs \bar{A}$ and by assumption we have $A\leqs B$. By Lemma \ref{mixedGenericStrong} this gives us $B\leqs D$ and $\bar{A}\leqslant D$.

Because $M$ is universal with respect to $\leqslant$, there is a $\leqslant$-strong embedding $F_0:D\rightarrow M$ with $F_0\upharpoonright \bar{A} = Id_{\bar{A}}$. Identify the structure $D$ with its image under $F_0$. Because $D\leqslant M$ there are no externally witnessed $S$-cliques in $D$ and so the $S$-diagram of $D$ remains the same within $M$, in particular the $S$-diagram of $F_0[B]$ as a substructure of $M$ is identical to that of $B$. By Proposition \ref{StrongImpliesLeqs} and $D\leqslant M$ we have $D\leqs M_s$. By transitivity and $B\leqs D$ we then have $B\leqs M_s$. In conclusion, we have found the embedding we were looking for, $F_0\upharpoonright B$.
\end{proof}

\begin{Cor}\label{homogeneity} $M_s$ is homogeneous with respect to strong embeddings.
\end{Cor}
\begin{proof}
Say $f:A\rightarrow B$ is a finite partial $S$-isomorphism with $A,B\leqs M_s$. Let $m_1,m_2,...$ be an enumeration of $M_s$.
\\Define $f_0 = f$ and denote $A_n = Dom(f_n),B_n = Rng(f_n)$.

Assume $f_{2n}$ is a given partial isomorphism with $A_{2n},B_{2n}\leqs M_s$. Choose $A_{2n}\cup\set{m_n}\subseteq \bar{A}\subseteq M_s$ such that $\bar{A}\leqs M_s$. Because $A\leqs M_s$ it is true that in particular $A\leqs \bar{A}$. Consider the diagram which is $\bar{A}$ after renaming every element $a\in A$ to $f(a)$ and call it $\bar{B}$. Note that $B\leqs \bar{B}$ and so by Proposition \ref{weakHomogeneity} there is a strong embedding $g:\bar{B}\rightarrow M_s$ with $g\upharpoonright B = Id_B$. Define $f_{2n+1} = f_{2n}\cup g\upharpoonright (\bar{A}\setminus A)$, it is a partial isomorphism with both domain and range strong in $M_s$.

Go about defining $f_{2n+2}$ from $f_{2n+1}$ the same way, only now adding $m_n$ to the range rather than the domain.

Take $\bar{f} = \bigcup_{i=0}^{\infty}f_i$, it is an automorphism of $M_s$ extending $f$ as required.
\end{proof}

\begin{Cor}\label{properReduct} $M_s$ is a proper reduct of $M$. Namely, the relation $R$ cannot be recovered from the relation $S$.
\end{Cor}

\begin{proof}
We show that there is an automorphism of $M_s$ which is not an automorphism of $M$. Consider the $R$-structure $A = \set{a,b,c}$ with $R_A = \emptyset$ and the $R$-structure $B = \set{1,2,3}$ with $R_B = \set{(1,2,3)}$. As $\emptyset \leqslant A,B$, both $A$ and $B$ are strongly embeddable into $M$. Identify the structures $A$ and $B$ with their images in $M$. So $A,B\leqslant M$ and therefore $A,B\leqs M$. The map $(a,b,c)\mapsto (1,2,3)$ is a partial $S$-isomorphism and so by homogeneity of $M_s$, extends to $\sigma\in Aut(M_s)$. As the map $(a,b,c)\mapsto (1,2,3)$ is \textbf{not} a partial $R$-isomorphism, $\sigma\notin Aut(M)$ and so the relation $R$ can not be defined using $S$.
\end{proof}

Now we have everything we need in order to characterize $M_s$ as the fra\"{i}ss\'e limit $\fC_s$.

\begin{theorem} $M_s$ is the fra\"{i}ss\'e limit of $\fC_s$
\end{theorem}

\begin{proof}
By Theorem \ref{Criterion}, $\fC_s$ is the age of $M_s$. By Proposition \ref{homogeneity}, $M_s$ is homogeneous with respect to strong embeddings. As discussed before, there is a single countable structure up to isomorphism with these qualities and so $M_s$ must be it.
\end{proof}

\section{The Theory of $M_s$}\label{Mstheory}
\begin{definition}
Define $F_n$ to be the $S$-structure with $n$ points and $C(F_n) = \emptyset$. Obviously $F_n\in \calC_s$
\end{definition}

\begin{definition}
We say an infinite $S$-structure $N$ is \emph{rich} if it has the following properties:
\begin{enumerate}
\item $\calC_s$ is the class of finite substructures of $N$
\item Let $A\subseteq N$ and $A\leqs B$ for some $B\in \calC_s$, then there is an embedding $f:B\rightarrow N$ such that $f\upharpoonright A = Id_A$.
\item For all $n\in \N$ there is a strong embedding of $F_n$ into $N$
\end{enumerate}

Note that 1 and 2 are first-order properties in the language $\set{S}$
\end{definition}

\begin{obs}\label{MsRich}
$M_s$ is rich.
\begin{enumerate}
\item By Theorem \ref{Criterion}
\item Denote $A' = cl(A)$. Let $D$ be the simple amalgam of $A'$ and $B$ over $A$. By definition, $A'\leqs D$ and so $D$ is strongly embeddable into $M_s$ over $A$. The embedding restricted to $B$ is the one sought after.
\item $F_n\in \calC_s$ and so by the property of $M_s$ as the fra\"{i}ss\'e limit of $\fC_s$, $F_n$ is strongly embeddable into $M_s$.
\end{enumerate}
\end{obs}

\begin{lemma}\label{independence}
Let $A\leqs N$. Denote $d = d_s(A) = \dos(A)$.
\\$(i)$ For $b_1,b_2\in N$, if $d_s(A,b_1) = d = d_s(A,b_2)$ then $d_s(A,b_1,b_2) = d$.
\\$(ii)$ If $N$ is rich then there is a point $b\in N$ such that $d < d_s(A,b) = d + 1$ and $A\cup\set{b}\leqs N$.
\end{lemma}

\begin{proof}
$(i)$ As $d_s(A,b_1,b_2) \geq d_s(A) = d$, it is enough to show $d_s(A,b_1,b_2) \leq d$. Let $B_i\subseteq N$ be such that $A\cup\set{b_i}\subseteq B_i$ and $\dos(B_i) = d$. As $A\cup\set{b_1,b_2}\subseteq B_1\cup B_2$ it is enough to prove $\dos(B_1\cup B_2) \leq d$.

$B_1,B_2\leqs N$ and so by Corollary \ref{intersectionOfStrongIsStrong} $B_1\cap B_2\leqs N$. By Lemma \ref{dosFreeAmalgam}, $0 = \dosover{B_1}{B_1\cap B_2} \geq \dosover{B_1}{B_2}$ so $\dos(B_1\cup B_2) \leq \dos(B_2) = d$.

\medskip
$(ii)$ Identify $F_{d+1}$ with a strong embedding of it into $N$. If $d_s(A,b) = d$ for all $b\in F_{d+1}$ then by a simple induction using $(i)$ we have $d_s(A,F_{d+1}) = d$. But that is a contradiction to the fact $d_s(F_{d+1}) = d + 1$, and so there must be some $b\in F_{d+1}$ such that $d < d_s(A,b) = d + 1$.

Obviously $\dos(A,b) = d+1$ and so $A\cup\set{b}\leqs N$.
\end{proof}

\begin{lemma}\label{StrongBackAndForth}
Let $N_1$ and $N_2$ be rich countable $S$-structures. Let $A\leqs N_1$, $B\leqs N_2$ be finite and let $f_0:A\rightarrow B$ a partial-isomorphism. Then for any point $a\in N_1$ there is a partial-isomorphism $f_0\subseteq f$ with $a\in Domf$, $Domf\leqs N_1$ and $Imf\leqs N_2$.
\end{lemma}

\begin{proof}
Denote $d= \dos(A) = \dos(B)$. Consider $A' = cl(A\cup\set{a})$. Note that $A\leqs A'$ because $A\leqs N_1$ and so $d\leq \dos(A')$.

If $d < \dos(A')$ then $A' = A\cup\set{a}$ and $a$ is not a member of any cliques in $A'$. Choose some $b\in N_2$ as guaranteed by Lemma \ref{independence} and define $f = f_0\cup\set{(a,b)}$. The element $b$ is not a member of any cliques in $B\cup\set{b}$ and so $f$ is a partial isomorphism as in the statement of the lemma.

If $d = \dos(A')$ then by richness of $N_2$ there is an embedding of $f:A'\rightarrow N_2$ with $f\upharpoonright A = B$. Denote $B' = Imf$, $f$ is a partial isomorphism between $A'$ and $B'$. Because $d = d_s(B) \leq d_s(B') \leq \dos(B') = d$ we have $\dos(B') = d_s(B')$ and $B'\leqs N_2$. So $f$ is as described in the statement of the lemma.
\end{proof}

\begin{lemma}\label{RichIsomorphic}
Any two rich countable $S$-structures $N_1$, $N_2$ are isomorphic. In particular, any countable rich structure is isomorphic to $M_s$
\end{lemma}

\begin{proof}
$\esi N_1,N_2$ and $f_0 = \emptyset$ is a partial isomorphism. Continue by standard back and forth between strong substructures using Lemma \ref{StrongBackAndForth} to obtain an isomorphism.
\end{proof}

\begin{Cor}
$M_s$ is a saturated model of $Th(M_s)$.
\end{Cor}

\begin{proof}
By the above lemma, $M_s$ is isomorphic to any elementary extension of itself  and therefore realizes all types in $S_n(Th(M_s))$ for any $n\in\N$. Since $M_s$ is countable, this means that $S_n(T)$ is countable for any $n\in \N$ and so there exists a saturated countable model of $Th(M_s)$. This model is of course rich and therefore, by the above lemma, is isomorphic to $M_s$. So $M_s$ is saturated.
\end{proof}

We now suggest an axiomatisation of $Th(M_s)$.
\begin{definition}
Define $T_s$ to be the first-order theory in the language $\set{S}$ that states the following:
\begin{enumerate}
\item $\calC_s$ is the class of finite substructures of $N$
\item Let $A\subseteq N$ and $A\leqs B$ for some $B\in \calC_s$, then there is an embedding $f:B\rightarrow N$ such that $f\upharpoonright A = Id_A$.
\item For all $n,m\in \N$ there is an embedding of $F_n$ into $N$ such that it is strong in any subset of $N$ of size $m$.
\end{enumerate}
\end{definition}

\begin{remark*}
In the above axiomatisation, the axiom scheme $3$ is implied by $1$ and $2$ and is therefore redundant. The axiom scheme is included nonetheless, since it makes the following argument simpler and clearer.
\end{remark*}

\begin{theorem}
$T_s$ is a complete theory and is an axiomatisation of $Th(M_s)$.
\end{theorem}

\begin{proof}
Let $n\in \N^{\geq 1}$. Consider $\bar{f}$ an enumration of $F_n$. Say an $S$-structure $A\in \calC_s$ is \emph{n-bad} if it is a superstructure of $F_n$ with $\dos(A) < n$. Denote $\varphi_{A}(\bar{x})$ the atomic diagram of a bad structure $A$ over the $n$-tuple $\bar{f}$. Define $p_n = \setarg{\neg \varphi_A(\bar{x})}{A\in \calC_s \text{ is n-bad}}$.

Let $N\models T_s$, without loss of generality $N$ is countable (otherwise take an elementary substructure). By article 3 of $T_s$, for any $n$ there is a realization of any finite $\Delta\subseteq p_n$ in $N$ and so $p_n$ is finitely satisfiable in $N$. Then there exists $N_0$, a countable elementary extension of $N$ realizing $p_n$ for all $n\in \N$. As $N_0$ is countable and rich, by Lemma \ref{RichIsomorphic} it is isomorphic to $M_s$, and therefore $Th(N) = Th(M_s)$. So $T_s\models Th(M_s)$ and is therefore complete.
\end{proof}

\begin{obs}
$T_s$ has quantifier elimination to the level of boolean combinations of formulas of the form $\exists \bar{x}(\phi(\bar{x},\bar{y}))$ where $\phi$ is a finite $S$-diagram.

Negations of formulas of this form are enough to imply that a set is strongly embedded. We have already seen that the closure of a set implies its complete type. Therefore using formulas of this form one can identify the closure of a tuple and imply its complete type.
\end{obs}

\section{The Pre-Geometry of $M_s$}\label{Mspregeometry}

We fix a first order language for the class of pre-geometries. The language is $LPI = \set{I_n}_{n\in \N}$ where $I_n$ is an $n$-ary relation symbol. We interpret a pre-geometry $(X,cl)$ as a structure in the language by taking $I_n^{X}$ to be the set of independent $n$-tuples of $X$. An infinite set is independent if and only if all its finite subsets are, and so this information suffices to define the pre-geometry.

Recall that a pre-geometry has an associated dimension function $d$. The dimension $d$ for a general set $A$ is defined to be the maximal cardinality of an independent subset of $A$. Thus, a finite set $A$ is independent iff $|A| = d(A)$.

Observe then, that a pre-geometry can be recovered from its finite independent sets or from its associate dimension function and vice versa.

\begin{remark}
Recall that in our context, an amalgamation class is a class of finite structures, $\calA$, with some class of distinguished embeddings, say $\mathfrak{M}_\calA$. In this section we depict an amalgamation class by $(\calA,\strong)$. By $\strong$ we mean the class of identity embeddings in $\mathfrak{M}_\calA$. The class $\mathfrak{M}_\calA$ may be recovered by taking the closure under composition of $\strong$ and the class of isomorphisms between structures in $\calA$.

In this section we also assume that all amalgamation classes have HP in relation to inclusion. That is, if $A\subseteq B\in \calA$ then $A\in \calA$.
\end{remark}

\begin{definition}
Let $(\calA,\strong)$ be an amalgamation class with a generic structure $M_\calA$. Let $d_{0\calA}:\calA\rightarrow \N^{\geq 0}$, we say $d_{0\calA}$ is a \emph{pre-dimension} function for $\calA$ if the following properties hold:
\begin{enumerate}
\item
$d_{0\calA}(\emptyset) = 0$
\item
$d_{0\calA}(A) \leq 1$ for all $A\in\calA$ with $|A| = 1$
\item
$d_{0\calA}(A\cup B) \leq d_{0\calA}(A) + d_{0\calA}(B) - d_{0\calA}(A\cap B)$ \hfill (Submodularity)
\item
$A\strong B \iff min\setarg{d_{0\calA}(D)}{A\subseteq D\subseteq B} \geq d_{0\calA}(A)$
\end{enumerate}

For $A\subset M_\calA$ finite, define $d_\calA(A) = min\setarg{d_{0\calA}(B)}{A\subseteq B\subseteq M \text{ finite}}$. We say $d_\calA$ is the \emph{dimension function} associated with $d_{0\calA}$.
\end{definition}

\begin{fact}
For an amalgamation class $(\calA,\strong)$ and a pre-dimension function $d_{0\calA}$ as above, the dimension function $d_\calA$ is the dimension function of a pre-geometry on $M_\calA$.
\end{fact}

\begin{definition} Let $(\calA,\strong)$ be an amalgamation class with a generic structure $M_\calA$ and a pre-dimension function $d_{0\calA}:\calA\rightarrow \N^{\geq 0}$.

\begin{itemize}
\renewcommand{\labelitemi}{$-$}

\item
Let $X\in \calA$ be a finite dimensional subset of $N$. We say a point $a\in M_\calA$ is \emph{dependent} on $X$ if $d_{\calA}(X,a) = d_{\calA}(X)$.

\item
We say a finite set $X\subseteq M_\calA$ is \emph{dependent} if there is some $x\in X$ such that $x$ is dependent on $X\setminus\set{x}$.

\item
Let $A\subseteq M_\calA$. We say that a set $X\subseteq A$ is \emph{$d$-closed} in $A$ if for any $a\in A$, if $a$ depends on $X$ then $a\in X$. The $d$-closure of a set $X$ in $A$ is the smallest $d$-closed set in $A$ containing $X$.
\end{itemize}
\end{definition}

\begin{definition}\label{isoext} For two amalgamation classes $(\calC_1,\strong)$ and $(\calC_2,\strong)$ we say the classes have the \emph{Isomorphism Extension Property} and denote $\calC_1 \isoext \calC_2$ if the following statement holds.
\begin{itemize}
\item[(*)]
Suppose $A_1\in \calC_1$, $A_2\in \calC_1$ and $f_0: PG(A_1)\rightarrow PG(A_2)$ is an isomorphism of pre-geometries, and $A_1\strong B_1\in \calC_1$ is finite. Then there is some $B_2\in \calC_2$ and an isomorphism $f:PG(D_1)\rightarrow PG(D_2)$ extending $f_0$.
\end{itemize}
\end{definition}

In their paper \cite{DavidMarcoTwo}, Evans and Ferreira show using a back-and-forth argument that if $(\calC_1,\strong)$ and $(\calC_2,\strong)$ are two amalgamation classes with $\calC_1\isoext \calC_2$ and $\calC_2\isoext \calC_1$, then the pre-geometries\footnote{In fact, Evans and Ferreira do this for geometries rather than pre-geometries, but the proof is identical.} of their respective generic structures $M_1$ and $M_2$ are isomorphic. The proof can be found in \cite{DavidMarcoTwo}, Lemma 2.3.

\begin{lemma} \label{dimensionleq}
Let $(\calC,\strong)$ and $(\calD,\strong)$ be amalgamation classes whose dimension functions originate in pre-dimension functions. Let $C\strong \bar{C}\in \calC$ and $D\strong \bar{D}\in \calD$ with a function $f:PG(\bar{C})\rightarrow PG(\bar{D})$ such that $f_0 = f\upharpoonright C$ is an isomorphism of pre-geometries between $PG(C)$ and $PG(D)$. Let $X\subseteq \bar{C}$ and let $\bar{X}$ be the $d$-closure of $X$ in $\bar{C}$. Then if $d_0(\bar{X}/\bar{X}\cap C) \geq d_0(f[\bar{X}]/f[\bar{X}]\cap D)$, then $d(X)\geq d(f[X])$.
\end{lemma}

\begin{proof}
We prove several facts:
\begin{enumerate}
\item $\bar{X}\cap C$ is $d$-closed in $C$.

\item $f[\bar{X}]\cap D$ is $d$-closed in $D$.

\item If a set $A$ is $d$-closed in $B$ then $d(A,B) = d_0(A)$

\item $d_0(\bar{X}\cap C) = d_0(f[\bar{X}]\cap D)$.
\end{enumerate}

\begin{proof}
\begin{enumerate}
\item Let $c\in C$ depend on $\bar{X}\cap C$. Then $c$ depends on $\bar{X}$ and since $\bar{X}$ is $d$-closed we have $c\in \bar{X}$. Therefore $c\in \bar{X}\cap C$ and $\bar{X}\cap C$ is $d$-closed in $C$.

\item Since $f_0$ is an isomorphism of pre-geometries, a point $d\in D$ depends on $f[\bar{X}]\cap D$ iff $f^{-1}(d)$ depends on $f^{-1}[f[\bar{X}]\cap D] = \bar{X}\cap C$. By the previous item, if $d\in D$ depends on $f[\bar{X}]\cap D$ then $f^{-1}(d)\in \bar{X}\cap C$ and so $d = f(f^{-1}(d)) \in f[\bar{X}]\cap D$. Thus, $f[\bar{X}]\cap D$ is $d$-closed in $D$.

\item Note that for $Y\subseteq B$, if $d_0(Y/A) \leq 0$ then $Y$ depends on $A$. Since $A$ is $d$-closed, $Y\subseteq A$ and so $d_0(Y/\bar{X}) = 0$. So $A \strong B$ and $d(A,B) = d_0(A)$.

\item For some $Y\in C$, $d(Y, C)$ is the cardinality of the smallest independent subset of $Y$. Therefore, $d(Y)$ is uniquely determined by the pre-geometry of $C$. Since $f_0$ is an isomorphism of pre-geometries between $C$ and $D$ we then have $d(\bar{X}\cap C, C) = d(f[\bar{X}]\cap D, D)$. By the previous items and the last equality, $d_0(\bar{X}\cap C) = d(\bar{X}\cap C, C) = d(f[\bar{X}]\cap D, D) = d_0(f[\bar{X}]\cap D)$
\end{enumerate}
\end{proof}

Since $\bar{X}$ is the $d$-closure of $X$, we have $d(\bar{X}) = d(X)$. By item $3$ and the fact $\bar{X}$ is $d$-closed in $\bar{C}$ we have $d(\bar{X}) = d_0(\bar{X})$.

Now assume $d_0(\bar{X}/\bar{X}\cap C) \geq d_0(f[\bar{X}]/f[\bar{X}]\cap D)$. By item $4$, it must be that $d_0(f[\bar{X}])\leq d_0(\bar{X})$. Since $f[X]\subseteq f[\bar{X}]$ this means $d(f[X]) \leq d_0(f[\bar{X}]) \leq d_0(\bar{X})$. Combining with $d_0(\bar{X}) = d(\bar{X}) = d(X)$ we have $d(X)\geq d(f[X])$ which concludes the proof.
\end{proof}

\begin{remark*}
In this section the relation $R$ is not necessarily symmetric and each $R$-tuple is composed of three distinct elements. Also, in this section $M$ denotes Hrushovski's construction for such a relation $R$ and $\calC$ denotes the age of $M$.
\end{remark*}

We wish to show that the pre-geometry of $M_s$ is isomorphic to that of $M$. We will do so as described in the discussion following Definition \ref{isoext}.

\begin{lemma} \label{CCextendstoCCS}
$\CC \isoext \CCS$
\end{lemma}

\begin{proof}
Let $A\in \CC$ and $D\in \CCS$ be finite with $f_0:PG(A)\rightarrow PG(D)$ an isomorphism. Let $A\strong B$ be finite. Define an $S$-structure $E$ with:

\[E = D\cup \setarg{e_b}{b\in B\setminus A}\]
\[C(E) = C(D)\cup \setarg{\set{e_{b_1},e_{b_2},e_{b_3}}}{(b_1,b_2,b_3)\in R(B)\setminus R(A)}\]

where the pre-multiplicity function is defined
\[\mult_E:C(E)\rightarrow \set{1,2,3}\]
\[\mult_E(C) =
	\left\{
		\begin{array}{ll}
			\mult_{D}(C) & \text{if } C\in C(D)\\
			|\setarg{(b_1,b_2,b_3)\in R_B\setminus R_A}{b_1,b_2,b_3\in C}| & \text{otherwise}
		\end{array}	
	\right.\]

Define $f = f_0\cup \setarg{(b,e_b)}{b\in B\setminus A}$. We claim that $D\leqs E\in \CCS$ and that $f$ is an isomorphism between $PG(B)$ and $PG(E)$

\medskip
\noindent\textbf{Claim 1.} Let $X\subseteq B$ and denote $Y = f[X]$, then $d_0(X/X\cap A) = \dos(Y/Y\cap D)$.

\begin{proof}

\begin{align*}
&\dosover{Y}{Y\cap D} =
\\&|Y|-|Y\cap D| - \sum_{C\in C(Y/Y\cap D)} \mult_E(C)\cdot |C|_* - \sum_{C\in C^Y(Y\cap D)} \mult_E(C)\cdot |C\setminus D| =
\\&|Y|-|Y\cap D| - \sum_{C\in C(Y/Y\cap D)} \mult_E(C) =
\\&|X|-|X\cap A| - r(X/X\cap A) = d_0(X/X\cap A)
\end{align*}
\end{proof}

An immediate conclusion of Claim 1 is that $D\leqs E$. By transitivity also $\emptyset \leqs E$ and so $E\in \CCS$.

\medskip
\noindent\textbf{Claim 2.} Let $X\subseteq B$ and denote $Y = f[X]$, then $d(X) = d_s(Y)$.
\begin{proof}
Let $\bar{X}$ be the $d$-closure of $X$ in $B$. By Claim 1, $d_0(\bar{X}/X\cap A) = \dosover{f[\bar{X}]}{f[\bar{X}]\cap D}$ and so by Lemma \ref{dimensionleq} we have $d(X) \geq d_s(Y)$.

Let $\bar{Y}$ be the $d$-closure of $Y$ in $E$. By Claim 1 we have $d_0(\bar{Y}/\bar{Y}\cap D) = \dosover{f^{-1}[\bar{Y}]}{f^{-1}[\bar{Y}]\cap A}$ and so by Lemma \ref{dimensionleq} we have $d_s(Y) \geq d(X)$.

%Denote $X_0 = X\cap A$ and $Y_0 = Y\cap D$. Note that by isomorphism of pre-geometries, a point $a\in A$ is dependent on $X_0$ iff $f_0(a)$ is dependent on $Y_0$. Thus, we may assume without loss of generality that $X_0$ is $d$-closed in $A$, meaning any point in $A$ dependent on $X_0$ is already in $X_0$. This assumption does not restrict generality because it does not change the dimension of either $X_0$ or $Y_0$.
%
%Since $X_0,Y_0$ are $d$-closed we have in particular $d_0(X_0) = d(X_0)$ and $\dos(Y_0) = d_s(Y_0)$. By isomorphism of pre-geometries we have $d(X_0) = d_s(Y_0)$ and so $d_0(X_0) = \dos(Y_0)$. Combining this equality with Claim 1, we have that for any $X\subseteq Z\subseteq B$ we have $d_0(Z) = \dos(f[Z])$. Symmetrically, for any $Y\subseteq Z\subseteq E$ we have $d_0(f^{-1}[Z]) = \dos(Z)$. By definition of dimension as a minimum, it is the case that $d(X) = d_s(Y)$.
\end{proof}

A dimension function uniquely determines the pre-geometry and so by Claim 2, $f$ is an isomorphism of pre-geometries.

\end{proof}

\begin{lemma}\label{distinctPairs}
Let $E\in \CCS$. There exists an injective function $f:C(E)\rightarrow [E]^2$ such that $f(C)\subseteq C$ for any $C\in C(E)$.
\end{lemma}

\begin{proof}

Take a choice of pairs $f$ such that the number of instances of $f(C) = f(D)$ with $C\neq D$ is minimal, we call such an instance a \emph{collision}. For a clique $C\in C(E)$ define $C^* = \setarg{D\in C(E)}{f(D)\subseteq C}$.

Assume that $f(C) = f(D)$ for two distinct cliques in $C(E)$. Define $X_0 = C$ and $C_0 = \set{C}$. If there is some pair $\set{c_1,c_2}\in [C]^2\setminus Im(f)$ then we can simply alter $f$ so that $f(C) = \set{c_1,c_2}$ and prevent a collision, so by minimality of $f$ we have $[C]^2\subseteq Im(f)$.

Define $C_i = \setarg{C^*}{C\in C_{i-1}}$ and $X_i = \bigcup C_i$. Let $D\in C_i\setminus C_{i-1}$ and let $D_0\in C_{i-1}$ be a clique such that $D\in D_0^*$. If there is some pair $\set{d_1,d_2}\in [D]^2\setminus Im(f)$, then we can alter $f$ so that $f(D) = \set{d_1,d_2}$. By minimality this cannot prevent a collision, and so now $[D_0]^2\setminus Im(f)\neq \emptyset$. Continuing this way we can find a choice of pairs $f$ with the same number of collisions such that $[C]^2\setminus Im(f) \neq \emptyset$, which as explained is a contradiction to minimality. So for any $n\in\N$ and for any $D\in C_n$ we have $[D]^2\subseteq Im(f)$.

Since $C(E)$ is finite, the process stabilizes and there is some $n\in\N$ such that $C_n = C_k$ for all $n < k$. Consider $X = X_n$. For any $x\in X\setminus X_0$, there is some $i$ such that $x\in X_{i+1}\setminus X_i$ and then there is some pair $\set{x,y}\in ([X_{i+1}]^2\setminus [X_i]^2)\cap Im(f)$. Along with the fact that $[X_0]^2\cap Im(f)$ has at least $|X_0|$ elements, we get that $[X_n]^2\cap Im(f)$ is of size at least $|X|$.

Recall that there is also a collision in $X$, which means $|C(X_n)| > |X|$ and consequently $\dos(X) < 0$, a contradiction. So in conclusion, there are no collisions in $f$ and the lemma is proven.
\end{proof}

\begin{definition}\label{emulation}
Let $D\in \CCS$ and $A\in \CC$ with $f_0:PG(D)\rightarrow PG(A)$ an isomorphism. Assume without loss of generality that the universe of $A$ is $\setarg{b_e}{e\in D}$ and that $f_0(e) = b_e$ for any $e\in D$.

Let $D\leqs E$ be a superstructure of $D$. Let $\fullenum{E}$ be a full-enumeration of $C(E)$. For any $C\in\fullenum{E}$, fix a pair of distinct elements $(x_C,y_C)\in C^2$ where if $C\in C^E(D)$ then also $x_C,y_C\in (C\cap D)^2$. Choose the pairs to be distinct (as ordered pairs).

We now define an $R$-structure $B$. The universe of $B$ is
\[B = A\cup \setarg{b_e}{e\in E\setminus D}\]

Let $f = f_0\cup \setarg{(e,b_e)}{e\in E\setminus D}$ and define
\[R(B) = R(A)\cup \bigcup_{C\in C^E(D)} f[C\setminus D]\times \set{b_{x_C}} \times \set{b_{y_C}}\cup \bigcup_{C\in C(E/D)} f[C]\times \set{b_{x_C}}\times\set{b_{y_C}}\]

We say $B$ is an \emph{emulation} of $E$ over $A$.
\end{definition}

\begin{remark} In the definition above it is possible to choose the pairs $(x_C,y_C)$ to be distinct. Simply use Lemma \ref{distinctPairs} on the structure $E$ where the cliques in $C^E(D)$ are restricted to $D$.
\end{remark}

In the definition of an emulation, the images of the elements $x_C,y_C$ in the structure $B$, in a way, witness the fact that the image of the clique $C$ is a dependent set of dimension $2$. The elements are taken from within the universe of the clique in order to assure the dimension does not increase.

\begin{lemma} \label{CCSextendstoCC}
$\CCS \isoext \CC$
\end{lemma}

\begin{proof}
Let $D\in \CCS$ and $A\in \CC$ with $f_0:PG(D)\rightarrow PG(A)$ an isomorphism. Let $D\leqs E$ be a superstructure of $D$. Let $B$ be an emulation of $E$ over $A$ as in the above definition. Let $\fullenum{E}$ and the pairs $(x_C,y_C)$ be as in the definition of an emulation.

Define $f = f_0\cup \setarg{(e,b_e)}{e\in E\setminus D}$. We claim that $A\strong B\in \calC$ and that $f$ is an isomorphism between $PG(E)$ and $PG(B)$. As in the proof of the previous lemma, it is enough to show that for any $X\subseteq E$ we have $d_s(X) = d(f[X])$.

Let $X\subseteq E$ be arbitrary, denote $Y = f[X]$.

\medskip
\noindent\textbf{Claim 1.} $d_s(X) \geq d(Y)$

\begin{proof}
Let $\bar{X}$ be the $d$-closure of $X$ in $E$. Denote $\bar{Y} = f[\bar{X}]$, $X_0 = \bar{X}\cap D$ and $Y_0 = \bar{Y}\cap B$.

Let $\fullenum{\bar{X}}$ be a full-enumerations of $C(\bar{X})$ and fix an injection $g:\fullenum{\bar{X}} \rightarrow \fullenum{E}$ with $g(C)\cap \bar{X} = C$. Note that if $C\in \fullenum{\bar{X}}$, then by the fact $\bar{X}$ is $d$-closed we have $g(C)\subseteq \bar{X}$. So in fact, $\fullenum{\bar{X}}$ is partial to $\fullenum{E}$ and so for any $C\in \fullenum{\bar{X}}$ the notation $x_C,y_C$ is meaningful. Also, for any $C\in\fullenum{\bar{X}}$ we have $x_C,y_C\in \bar{X}$ and therefore $b_{x_C},b_{y_C}\in \bar{Y}$.

Let $\fullenum{1}$ and $\fullenum{2}$ be full enumerations of $C^{\bar{X}}(X_0)$ and $C(\bar{X}/X_0)$ respectively. Assume some implicit injections from $\fullenum{i}$ to $\fullenum{\bar{X}}$ so that for $C\in\fullenum{i}$ the notation $x_C,y_C$ is defined. Now:

\begin{align*}
&s(\bar{X}) - s(X_0) =
\\&\sum_{C\in \fullenum{1}} |C\setminus X_0| + \sum_{C\in \fullenum{2}} |C|_* =
\\&\sum_{C\in \fullenum{1}} |f[C\setminus X_0]\times\set{b_{x_C}}\times\set{b_{y_C}}| + \sum_{C\in \fullenum{2}} |f[C]\times\set{b_{x_C}}\times\set{b_{y_C}}|_* =
\\&\sum_{C\in \fullenum{1}} |f[C\setminus X_0]\times\set{b_{x_C}}\times\set{b_{y_C}}| + \sum_{C\in \fullenum{2}} |(f[C]\setminus\set{b_{x_C},b_{y_C}})\times\set{b_{x_C}}\times\set{b_{y_C}}|\leq
\\ &r(\bar{Y}/Y_0)
\end{align*}

So $s(\bar{X}) - s(X_0) \leq r(\bar{Y}/Y_0)$ and since $|\bar{X}\setminus X_0| = |\bar{Y}\setminus Y_0|$ we consequently have $\dosover{\bar{X}}{X_0} \geq d_0(\bar{Y}/Y_0)$. By Lemma \ref{dimensionleq} this gives us $d_s(X) \geq d(Y)$.
\end{proof}

\noindent\textbf{Claim 2.} $d(Y)\geq d_s(X)$

\begin{proof}
Let $\bar{Y}$ be the $d$-closure of $Y$ in $B$. Denote $\bar{X} = f^{-1}[Y]$, $X_0 = \bar{X}\cap D$ and $Y_0 = \bar{Y} \cap A$. Note that because $f_0$ is an isomorphism and $Y_0$ is $d$-closed in $A$, $X_0$ is $d$-closed in $D$.

Let $\fullenum{\bar{X}(X_0)}$ and $\fullenum{\bar{X}/X_0}$ be full-enumerations of $C^{\bar{X}}(X_0)$ and $C(\bar{X}/X_0)$ respectively, and let $\fullenum{\bar{X}} = (\fullenum{\bar{X}(X_0)},\fullenum{\bar{X}/X_0})$ be a full-enumeration of $C(\bar{X})$. Fix an injection $g:\fullenum{\bar{X}}\rightarrow \fullenum{E}$ as in the previous claim.

Denote $\wtclq{\Y}{x,y} = \setarg{e\in \bar{Y}}{(e,x,y)\in R(\bar{Y})\setminus R(Y_0)}$. Note that if $\wtclq{\Y}{x,y} \neq \emptyset$ then $(x,y) = (b_{x_C},b_{y_C})$ for some unique $C\in \fullenum{E}$. Moreover, because $f^{-1}[\wtclq{\Y}{x,y}]\cup\set{x_C,y_C}\subseteq \bar{X}\cap C$, there is a unique $\wtclq{C}{x,y}\in \fullenum{\bar{X}}$ with $g(\wtclq{C}{x,y}) = C$. By construction, if $(x_1,y_1)\neq (x_2,y_2)$ then $\wtclq{C}{x_1,y_1} \neq \wtclq{C}{x_2,y_2}$.

Note that if $\set{x_C,y_C}\subseteq D$, then $C\cap D$ is dependent on $\set{x_C,y_C}$ in $D$. Therefore, for such $\wtclq{C}{x,y}$ (for some $x,y$ such that $\wtclq{\Y}{x,y} \neq \emptyset$), if $\set{x_C,y_C}\subseteq X_0$ it must be that $\wtclq{C}{x,y}\cap X_0 = g(\wtclq{C}{x,y})\cap D$. This is due to the fact $X_0$ is $d$-closed in $D$. Thus, $g(\wtclq{C}{x,y})\in C^E(D)$ implies that $\wtclq{C}{x,y}\in \fullenum{\bar{X}(X_0)}$.

If $\wtclq{C}{x,y}\in \fullenum{\bar{X}(X_0)}$ then by definition $g(\wtclq{C}{x,y})\in C^E(D)$. Denote $C= g(\wtclq{C}{x,y})$. By construction of the emulation we have $f^{-1}[\wtclq{\Y}{x,y}]\subseteq (C\setminus D)\cap \bar{X} = (\wtclq{C}{x,y}\setminus X_0)$. Since $f$ is a bijection, $|\wtclq{\Y}{x,y}| \leq |\wtclq{C}{x,y}\setminus X_0|$

If $\wtclq{C}{x,y}\in \fullenum{\bar{X}/X_0}$ then it must be that $g(\wtclq{C}{x,y})\in C(E/D)$. Denote $C= g(\wtclq{C}{x,y})$. By construction, $f^{-1}[\wtclq{\Y}{x,y}]\cup \set{x_C,y_C}\subseteq C\cap \bar{X} = \wtclq{C}{x,y}$. Since $f^{-1}[\wtclq{\Y}{x,y}]\cap \set{x_C,y_C} =    \emptyset$, $f$ is a bijection and $f^{-1}[\wtclq{\Y}{x,y}]$ is not empty, it must be that $|\wtclq{\Y}{x,y}| \leq |\wtclq{C}{x,y}|_*$.

We now compute:

\begin{align*}
r(\bar{Y}/Y_0) &= \sum_{(x,y)\in \bar{Y}^2} |\wtclq{\Y}{x,y} \times \set{x}\times\set{y}|
\\&= \sum_{(x,y)\in \bar{Y}^2} |\wtclq{\Y}{x,y}|
\\&= \sum_{\substack{(x,y)\in \bar{Y}^2\\\wtclq{C}{x,y}\in \fullenum{\bar{X}(X_0)}}} |\wtclq{\Y}{x,y}| + \sum_{\substack{(x,y)\in \bar{Y}^2\\\wtclq{C}{x,y}\in \fullenum{\bar{X}/X_0}}} |\wtclq{\Y}{x,y}|
\\&\leq \sum_{\substack{(x,y)\in \bar{Y}^2\\\wtclq{C}{x,y}\in \fullenum{\bar{X}(X_0)}}} |\wtclq{C}{x,y}\setminus X_0| + \sum_{\substack{(x,y)\in \bar{Y}^2\\\wtclq{C}{x,y}\in \fullenum{\bar{X}/X_0}}} |\wtclq{C}{x,y}|_*
\\&\leq \sum_{C\in \fullenum{\bar{X}(X_0)}} |C\setminus X_0| + \sum_{C\in \fullenum{\bar{X}/X_0}} |C|_* 
\\&= s(\bar{X}) - s(X_0)
\end{align*}

So $r(\bar{Y}/Y_0)\leq s(\bar{X}) - s(X_0)$ and since $|\bar{Y}\setminus Y_0| = |\bar{X}\setminus X_0|$ we consequently have $d_0(\bar{Y}/Y_0)\geq \dosover{\bar{X}}{X_0}$. By Lemma \ref{dimensionleq} this gives us $d(Y) \geq d_s(X)$.
\end{proof}

Combining the two claims, we have $d_s(X) = d(Y) = d(f[X])$. A dimension function uniquely determines the pre-geometry and so $f$ is an isomorphism of pre-geometries.

It is not hard to verify that in general $N_1\strong N_2$ iff $d(X,N_1) = d(X,N_2)$ for any $X\subseteq N_1$. Thus, the relation $\strong$ (or $\leqs$) is determined by the dimension function. By the fact $D\leqs E$ and the equivalence of the dimension functions, we have $A\strong B$. The fact $B\in \CC$ follows from $A\in \CC$ and transitivity of $\strong$.
\end{proof}

\begin{theorem}\label{PGMandMs}
The pre-geometries $PG(M)$ and $PG(M_s)$ are isomorphic
\end{theorem}

\section{The Symmetric Reduct of The General non-collapsed Structure}\label{symmetric}

\newcommand{\rz}[1]{r(#1)}
\newcommand{\rzsim}[1] {R_{\sym}(#1)}
\newcommand{\dz}[1]{d_0(#1)}
\newcommand{\dzover}[2]{d_0(#1/#2)}
\newcommand{\dzsim}[1]{d_{0}^{\sim}(#1)}
\newcommand{\dzsimover}[2]{d_{0}^{\sim}(#1/#2)}
\newcommand{\ddim}[1]{d(#1)}
\newcommand{\dsimdim}[1]{d^{\sim}(#1)}

\def \strong {\leqslant}
\def \simstrong {\leqslant_{\sim}}
\def \Rsim {R_{\sim}}
\def \Msim {M_{\sim}}
\def \order {\unlhd}
\def \Csim {\fC_\sim}

In this section we denote by $M$ the non-collapsed countable Hrushovski construction for a ternary relation $R$ (not necessarily symmetric) where each tuple is composed of distinct elements. We denote by $\calC$ the age of $M$.

\begin{definition}
Let
\[\Rsim(x_1,x_2,x_3):= \bigvee_{\sigma\in S_3} R(x_{\sigma(1)},x_{\sigma(2)},x_{\sigma(3)})\]
\end{definition}

We call $\Rsim$ the symmetrization of $R$. We wish to show that $\Msim$, the reduct of $M$ to $\Rsim$, is a proper reduct. Moreover, we will show that $\Msim$ is Hrushovski's non-collapsed construction for a symmetric ternary relation.

\begin{definition}
Let $A,B\subseteq N$ be  $R$-structures with $A,B$ finite. Define:
\\
$\rzsim{A} = \setarg{\set{a,b,c}\in {[A]}^{3}}{(a,b,c)\in R_A}$
\\
$\dzsim{A} = |A| - |\rzsim{A}|$
\\
$\dzsimover{B}{A} = \dzsim{B\cup A} - \dzsim{A}$
\\
$\dsimdim{A,N} = \max\setarg{\dzsim{B}}{A\subseteq B\subseteq N,\ B \text{ finite}}$
\\
$A\simstrong N$ if $\dsimdim{A,N} = \dzsim{A}$. We say $A$ is $\sim$-strong in $N$.

\smallskip
\noindent When the superstructure is $M$, we simply denote $\dsimdim{A}$.
\end{definition}

\begin{remark*} The relation $\simstrong$ is transitive. The proof is identical to that of the transitivity of $\strong$.
\end{remark*}

\begin{lemma} Let $A,B\subseteq M$ be finite. Then $\dzsimover{B}{A} \geq \dzover{B}{A}$
\end{lemma}

\begin{proof}
Without loss of generality, assume $A\subseteq B$. It is then enough to show that $|R(B)\setminus R(A)|\geq|\Rsim(B)\setminus \Rsim(A)|$. But this is trivial, as the map $(a,b,c)\mapsto \set{a,b,c}$ is surjective from $R(B)\setminus R(A)$ onto $\Rsim(B)\setminus \Rsim(A)$.
\end{proof}

\begin{Cor}\label{strongImpliesSimstrong}
$A\strong M \implies A\simstrong \Msim$ 
\end{Cor}

\begin{proof}
By definition, using the above lemma.
\end{proof}

\begin{Cor}\label{closedFiniteSuperset}
For any $A\subseteq M$ there is some $A\subseteq B\simstrong M$ with $B$ finite.
\end{Cor}

\begin{proof}
Take $B$ to be the self-sufficient-closure of $A$.
\end{proof}

\begin{definition}
We say an $R$-structure $A$ is \emph{symmetric} if for any $a_1,a_2,a_3\in A$ we have $A\models R(a_1,a_2,a_3) \rightarrow \bigwedge_{\sigma\in S_3} R(a_{\sigma(1)},a_{\sigma(2)},a_{\sigma(3)})$. We also call a symmetric $R$-structure an $\Rsim$-structure.

Define $\calC_\sim = \setarg{A}{A \text{ is a finite } \Rsim\text{-structure with }\emptyset\simstrong A}$.
\end{definition}

Observe that an $\Rsim$-reduct of an $R$-structure is an $\Rsim$-structure.
%We define the predimension and dimension of $\Rsim$-structures so that they coincide with the definitions given for $R$-structures.

%\begin{definition}
%Let $A,B\subseteq N$ be $\Rsim$-structures with $A$ and $B$ finite. Define:
%\\
%$\rzsim{A} = \setarg{\set{a,b,c}\in {[A]}^{\leq 3}}{(a,b,c)\in \Rsim^A}$
%\\
%$\dzsim{A} = |A| - |\rzsim{A}|$
%\\
%$\dzsimover{B}{A} = \dzsim{B\cup A} - \dzsim{A}$
%\\
%$\dsimdim{A,N} = \max\setarg{\dzsim{B}}{A\subseteq B\subseteq N,\ B \text{ finite}}$
%\\
%$A\simstrong N$ if $\dsimdim{A,N} = \dzsim{A}$. We say $A$ is $\sim$-strong in $N$.
%\end{definition}

\begin{definition}
Define $\Csim$ to be the category where the set of objects is $\calC_\sim$ and the morphisms are $\sim$-strong embeddings of $\Rsim$-structures.
\end{definition}

The category $\Csim$ has HP, JEP and AP, where the amalgam is simply the free join of two structures. The Fra\"{i}ss\'e limit of $\Csim$ is Hrushovski's non-collapsed construction for a symmetric ternary relation. We claim that $\Msim$ is in fact this very structure.

We fix $\order$, a total ordering of the elements of the universe.

\begin{definition} \label{RStructRep}
Let $A$ be an $R$-structure with $\Rsim$-reduct $A_\sim$. Let $B$ be an $\Rsim$-structure with $A_\sim\subseteq B$. We define $D$, the \emph{$R$-structure representation of $B$ over $A$}. The universe of $D$ is the set $B$ and the relation $R(D)$ is defined thus:
\[R(D) = R(A)\cup \setarg{(a,b,c)\in B^3\setminus A^3}{a\order b\order c,\ (a,b,c) \in R(B)}\]
\end{definition}

\begin{obs}\label{dObs}
\*
\begin{enumerate}
\item
Let $D$ be as in the definition above. Then $\dsimdim{A_\sim,B} = \ddim{A,D}$. In particular, if $A_\sim\simstrong B$ then $A\strong D$.
\item
Let $D$ be as in the definition above. The $\Rsim$-reduct of $D$ is exactly $B$. 
\end{enumerate}
\end{obs}

\begin{lemma}
The age of $\Msim$ is $\calC_\sim$.
\end{lemma}

\begin{proof}
Let $A_\sim\in \calC_\sim$. Let $D$ be the $R$-structure representation of $A_\sim$ over $\emptyset$. Since $\emptyset\simstrong A_\sim$, by Obesrvation \ref{dObs}.i we have $\emptyset\strong D$ and thus $D$ is strongly embeddable into $M$. By Lemma \ref{strongImpliesSimstrong} we have $D\simstrong\Msim$ and by Observation \ref{dObs}.ii the $\Rsim$ reduct of $D$ is exactly $A_\sim$. So in conclusion $A_\sim$ is $\sim$-strongly embeddable into $\Msim$.

Let $A_\sim\subseteq \Msim$. We know that $\emptyset\strong M$ and so by Corollary \ref{strongImpliesSimstrong} we have $\emptyset\simstrong \Msim$. In particular $\emptyset\simstrong A_\sim$ and so $A_\sim\in \calC_\sim$.
\end{proof}

\begin{lemma}\label{embeddingOver} Let $A\simstrong\Msim$ be finite. Let $B\in \Csim$ with $A\simstrong B$. Then there is some $\sim$-strong embedding $f:B\rightarrow\Msim$ with $f\upharpoonright A = Id_A$.
\end{lemma}

\begin{proof}
Consider $A$ as an $R$-substructure of $M$. Let $\bar{A}\subseteq M$ be the self-sufficient-closure of $A$. Let $D$ be the $R$-structure representation of $B$ over $A$. By Observation \ref{dObs}.i, $A\strong D$. Take $E$ to be the free join of $D$ and $\bar{A}$ over $A$. By the properties of a free join and $A\strong D$, we have $\bar{A}\strong E$. It is also evident that $B\simstrong E$, because $A\simstrong \bar{A}$. By the property of $M$ as a fra\"{i}ss\'e limit, $E$ may be strongly embedded into $M$ over $\bar{A}$, we assume the embedding is the identity map. By Lemma \ref{strongImpliesSimstrong} we have $E\simstrong \Msim$ and so by transitivity $B\simstrong\Msim$. Thus, $B$ is $\sim$-strongly embeddable over $A$.
\end{proof}

\begin{prop} Let $A,B\simstrong\Msim$ be finite and let $f:A\rightarrow B$ be a finite partial $\Rsim$-isomorphism, then $f$ extends to an automorphism of $\Msim$.
\end{prop}

\begin{proof}
A well-known back-and-forth argument between $\sim$-strong subsets, utilizing Lemma \ref{embeddingOver} and Corollary \ref{closedFiniteSuperset}.
\end{proof}

\begin{Cor}
The reduct $\Msim$ is a proper reduct of $M$.
\end{Cor}

\begin{proof}
Consider the two $R$-structures:

\[A = \set{a_1,a_2,a_3}\]
\[R_A = \set{(a_1,a_2,a_3),(a_1,a_2,a_3)}\]

\[B = \set{b_1,b_2,b_3}\]
\[R_B = \set{(b_1,b_2,b_3)}\]

Assume $A,B\strong M$ and therefore $A,B\simstrong\Msim$. Note that $\set{a_1,a_2}\simstrong A\simstrong\Msim$ and $\set{b_1,b_2}\simstrong B\simstrong\Msim$. By the above proposition, the partial isomorphism $f = \setarg{(a_i,b_i)}{i\in\set{1,2}}$ can be extended into an automorphism $f_0\in Aut(\Msim)$. But, $f_0\notin Aut(M)$ because $\set{a_1,a_2}\nleqslant M$ while $\set{b_1,b_2}\strong M$. Thus, $\Msim$ is a proper reduct of $M$.
\end{proof}

\begin{theorem}
The structure $\Msim$ is the fra\"{i}ss\'e limit of $\Csim$.
\end{theorem}

\begin{proof}
There is exactly one (up to isomorphism) countable structure with age $\calC_\sim$ that is homogeneous in relation to $\sim$-strong substructures. Since the age of $\Msim$ is $\calC_\sim$ and it is homogeneous in relation to $\sim$-strong substructures, $\Msim$ must be that structure.
\end{proof}

\begin{prop}\label{GMandMsim}\footnote{This is noted in less detail as Theorem 4.7 of \cite{DavidMarcoTwo}}
The geometry of $M$ is isomorphic to that of $\Msim$.
\end{prop}

\begin{proof} (sketch)
Definition \ref{isoext} and the discussion that follows are valid also for geometries instead of pre-geometries. So it is enough to show that $\calC$ and $\calC_\sim$ have the isomorphism extension property for geometries (denote $\overset{G}{\isoext}$) in both directions.

\noindent$\calC\overset{G}{\isoext}\calC_\sim$:

\noindent Let $A\in\calC$, any three points in $A$ with more than one relation on them are of dimension less than $2$ and so in $G(A)$, the geometry of $A$, they all collapse into the same point or are absent from the structure completely. Thus, for any finite $A\in \calC$ there is some naturally associated $B\in \calC$ with the same geometry as $A$ and $PG(B) = G(B)$. It is not hard to see that the symmetric reduct of $B$ has the same geometry as $B$ and therefore as $A$. This works exactly the same when done over some substructure $A_0$ with $A_0\strong A$.

\noindent$\calC_\sim\overset{G}{\isoext}\calC$:

\noindent Let $A\in\calC_\sim$. As in Definition \ref{RStructRep}, one may take an $R$-structure representation of $A$ over $\emptyset$ in order to receive some $D\in\calC$ with the same geometry as $A$. This can also be done over a structure (modulo simple technicalities due to the fact the universe of the geometry is a quotient of the structure).
\end{proof}

In fact, it seems that the pre-geometries of $M$ and $\Msim$ are isomorphic, but the proof is very technical and the claim is of no consequence to us.

\section{Appendix}\label{appendix}

\begin{lemma}\label{fiveclique}
Let $\emptyset\leqslant N$. If $A=\set{\otn{a}{5}}\subseteq N$ and $S(a_1,a_2,a_3)$ for any $a_1,a_2,a_3\in A$, then there exists a pair of elements $x,y\in N$ such that $(a,x,y)\in R_N$ for all $a\in A$.
\end{lemma}

\begin{remark*} This proof is just an exhaustion of all possible options. In spirit, it is identical to a computer program running over all possibilities. It is included for the sake of completeness.
\end{remark*}

\begin{proof} Assume $A\subseteq N$ contradicts the statement.
\\We first show that all pairs of witnesses for triplets in $A$ cannot be external to $A$.
\smallskip
\\\textbf{Claim} All pairs of witnesses for the clique $A$ are not contained within $A$.
\begin{proof}
Assume a pair of witnesses for some $S$-triplet in $A$ is contained in $A$, say $\set{a_1,a_2}$ witness $S(a_3,a_4,a_5)$, then there are already three known $R$ relations amongst elements of $A$ (namely $\set{a_3,a_4,a_5}\times\set{a_1}\times\set{a_2}$) and the dimension of $A$ can be decreased further by only 2.

Another set of internal witnesses reduces the dimension by at least 2 more, and so it is impossible for $A$ to be a clique because not all $S$-triplets are witnessed yet, and witnessing any unwitnessed triplet will further reduce the dimension. So there cannot be another internal pair of witnesses.

The relations $S(a_1,a_2,a_3), S(a_1,a_2,a_4),$ $ S(a_1,a_2,a_5)$ must be witnessed and as the witnesses are non-internal for all of these, adding them to $A$ not only does not raise the predimension (because the witnesses are related to $a_1,a_2$ and so at worst case they are exactly cancelled out), but actively lowers it by at least three for their relatedness to {$a_3,a_4,a_5$} (regardless of how many pairs of witnesses). In conclusion - we may assume there are no internal pairs of witnesses.
\end{proof}

Assume each pair of witnesses relates to only three elements, then we have ${5 \choose 3} = 10$ different pairs of witnesses associated with triplets.  Note that a single $R$ relation can be used for the generation of at most two triplets of the clique $A$ because there is no internal pair of witnesses. Also, the dimension of $A$ is at most $5$.

If no $R$ relation is used in two different triplets, then by adding all witnesses, we add at most $2\cdot 10 = 20$ elements and $3\cdot 10 = 30$ relations to $A$, and have dimension below zero.

If there is a relation being used in the generation of two distinct triplets $T_1,T_2$, then both triplets have exactly one internal witness and they must have the same external one, say $t$. Adding $t$ to $A$ adds $1$ element and $5$ relations, and so far we have dimension at most $1$. There can be only $4$ more triplets sharing relations with $T_1$ or $T_2$ which leaves us with at least $4$ other triplets. If none of these other triplets share a relation with any other triplet, then they reduce the dimension by at least $4$ more which brings it below zero. If there is a triplet with a shared relation amongst the remaining $4$, say $T_3$, then it has one internal witness and one external witness. Adding the external witness to our current structure reduces the predimension by at least $2$ more ($3$ relations and one new element), which brings it below zero.

So there must be a pair of witnesses that relates to four elements of $A$.

Say $T = \set{a_1,a_2,a_3,a_4}$ is witnessed by the same pair. So far we have dimension 3 at best. There are now ${4 \choose 2} = 6$ S-relations left to witness ($a_5$ with any other two elements). Assume again that all other pairs of witnesses relate to only three elements.

If the clique $T$ shares a relation with a triplet $T_1$, then $a_5$ is one of the witnesses for $T$ and so the current dimension is 2 at best. Both of the witnesses for $T_1$ already appear in the structure (The external witness for $T_1$ is the same as that of $T$) and it shares only one relation with $T$, which implies there are two previously unaccounted for relations in the structure, bringing us to dimension zero. Choose another triplet $T_2$, if one of its witnesses is previously unseen, then adding the pair of witnesses adds three relations to structure and the dimension goes below zero. If both witnesses are already in the structure, $T_2$ has at most one relation shared with $T$ and one relation shared with $T_1$ and so contains one extra relation unaccounted for, which also takes the dimension below zero. So $T$ does not share a relation with a triplet.

If there are two triplets $T_1,T_2$ sharing a relation, then as before adding their external witness reduces the dimension by $4$, bringing it below zero.

So no triplets share relations and by adding all witnesses we add at most $2\cdot 6 = 12$ elements and $3\cdot 6 = 18$ relations, which lowers the dimension below zero.

So again there must be a pair of witnesses (a different pair by assumption) relating to four elements, say $T' = \set{a_1,a_2,a_3,a_5}$ and the dimension is reduced to one (the two size $4$ cliques cannot share a relation because then they would have the same external witness and adding it reduces dimension below zero).

We are left with the relations $S(a_1,a_4,a_5),S(a_2,a_4,a_5),S(a_3,a_4,a_5)$ to take care of. Not all three of these relations can be witnessed by the same pair by our assumption, and so at least two new pairs are involved. Note that any of the remaining cliques cannot share a relation with either $T$ or $T'$, because then they would both have an internal witness and the same external witness which we already said cannot happen. If one of the cliques shares a relation with $T$ (without loss of generality) then one of the witnesses of $T$ is internal, and so the dimension is already zero, the unaccounted for relations from the aforementioned clique lower the dimension further. So none of the remaining cliques shares a relation with $T$ or $T'$. Whether they share a relation amongst themselves or do not share relations at all does not matter, as adding their witnesses the the structure lowers the dimension below zero.

So finally, in conclusion, any S-clique of size five has a common pair of witnesses.
\end{proof}

\begin{Cor}\label{morethanfiveclique}
Let $\emptyset\leqslant N$. Let $A\subseteq N$ with $|A| > 4$ and $S(a_1,a_2,a_3)$ for all $a_1,a_2,a_3\in A$ distinct, then there exists a pair of witnesses $x,y\in N$ such that $(a,x,y)\in R$ for all $a\in A$.
\end{Cor}

\begin{proof}
We prove by induction on the size of the clique $A$: For a set of size 5, that is exactly Lemma \ref{fiveclique}. Assume the Corollary is true for cliques of size $i$ with $i>4$, let $|A| = i+1$. Let $a_1,a_2\in A$ be distinct elements. Since $|A\setminus \set{a_1}| = i$, by assumption there is a pair $\set{x_1,y_1}$ such that $(a,x_1,y_1)\in R_N$ for all $a\in A\setminus\set{a_1}$. Similarly, there is a pair $\set{x_2,y_2}$ such that $(a,x_2,y_2)\in R_N$ for all $a\in A\setminus\set{a_2}$. Assume $\set{x_1,y_1} \neq \set{x_2,y_2}$. Without loss of generality $x_i \notin \set{x_{3-i},y_{3-i}}$.

If there is an instance of $\set{a_i,x_1,y_1} = \set{a_j,x_2,y_2}$ for some $a_i,a_j\in A$ distinct, then $y_1 = y_2\notin A$, $x_1 = a_i = a_1$ and $x_2 = a_j = a_2$. In such case there are at least $2i-1$ distinct $R$ relations amongst the $i+2$ elements of $A\cup\set{y_1}$, since $i>3$ this means $d_0(A\cup\set{y_1}) < 0$ which contradicts $\emptyset \leqslant A$. So such an instance does not exist, and there is no intersection between relations forming $A\setminus \set{a_1}$ and $A\setminus \set{a_2}$ as cliques.

Denote $A' = A\cup\set{x_1,y_1,x_2,y_2}$. We now have $|A'| \leq i+5$ and $r(A') \geq 2i$ and so $0 \leq d_0(A') \leq i+5 - 2i = 5-i \leq 0$. And so $|A'| = i+5$, meaning the witnesses are external to $A$, and $r(A') = 2i$, meaning the only relations in $A'$ are the ones we accounted for. Also, $A' \leqslant N$ and so any witnesses for $S$-triplets in $A'$ are elements of $A'$. Now, the only relation in $A'$ containing $a_1$ is $(a_1,x_2,y_2)\in R_N$ and the only relation in $A'$ containing $a_2$ is $(a_2,x_1,y_1)\in R_N$. because $A$ is a clique there is a pair $\set{x_3,y_3}$ with $(a_1,x_3,y_3),(a_2,x_3,y_3),(a,x_3,y_3)\in R_N$ for some $a\in A\setminus\set{a_1,a_2}$. Because $A' \leqslant M$, it must be true that $\set{x_3,y_3} \subseteq A'$. We then must have $\set{x_1,y_1} = \set{x_3,y_3} = \set{x_2,y_2}$ which is a contradiction because $\set{x_1,y_1} \neq \set{x_2,y_2}$. And so it cannot be true that $\set{x_1,y_1} \neq \set{x_2,y_2}$.

So $\set{x_1,y_1} = \set{x_2,y_2}$ and so the entire clique $A$ is witnessed by $\set{x_1,y_1}$.

By induction, the Corollary is true for all $i > 5$
\end{proof}

\begin{lemma}\label{gradualIsOutsourcingLemma}
Let $A\subseteq B$ be $R$-structures and let $K,L\subseteq C(B)$ mutually exclusive with $W_{K\cup L}\subseteq B$. Assume $K$ is non-empty. Then any gradual outsourcing $B_{|K|}$ of $K$ with core $(A,L)$ with respect to $L$ is isomorphic to $\outsource{B}{A}{K,L}$. We use the notation of Definition \ref{A^Kdefinition}.
\end{lemma}

\begin{proof} The isomorphism is in fact the identity map, after some elements of $B$ were relabelled during the gradual outsourcing process.

\noindent$\mathbf{B_{|K|}\subseteq \outsource{B}{A}{K,L}}:$ Throughout the outsourcing process, the only new elements added to the structure $B$ are elements of $\calW_K$. Also, the elements of $EW_{K,L}(A)$ were relabelled into elements of $\calW_K$. By definition an element $x\in EW_{K,L}(A)$ does not witness cliques in $L$ and so when the last clique the element witnesses in $K$ is outsourced, that element is relabelled.

\smallskip
\noindent$\mathbf{B_{|K|}\supseteq \outsource{B}{A}{K,L}}:$ Any element of $B\setminus EW_{K,L}(A)$ is not removed from the structure or relabelled at any stage. Let $z^C\in \calW_K$, say the clique $C$ was outsourced at the stage $i$, then $z^C\in B_{i+1}$. Since $\clqwit{C}{x^C,y^C}\notin I_{j}$ for $j>i$ the element $z^C$ will not be removed or relabelled, so we have also $z^C\in A_{|K|}$.

\bigskip
\noindent$\mathbf{R(B_{|K|})\subseteq R(\outsource{B}{A}{K,L})}:$ Let $r\in R(B_{|K|})$.

If $r\in R(B_{|K|})\setminus R_B$ then $r \in RW{\clqwit{C}{x^C,y^C}}$ for some $C\in K$ because these are the only new relations added to the structure, and so $r\in R(\outsource{B}{A}{K,L})$. So assume $r\in R_B$.

If $r\notin(R_A \cup \bigcup_{C\in T_0}RW(C))$ then $r$ was removed when the first clique was outsourced and so $r\notin R(B_{|K|})$.

If $r\in \bigcup_{C\in K}RW(C)$ then by mutual exclusivity $r\notin \bigcup_{C\in L}RW(C)$ and so $r$ is removed from the structure when the last clique that shares it is outsourced, so $r\notin R(B_{|K|})$.

The only remaining options are $r\in R_A\setminus \bigcup_{C\in K}RW(C)$ or $r\in \bigcup_{C\in L}RW(C)$ so it must be that $r\in R(\outsource{B}{A}{K,L})$.

\smallskip
\noindent$\mathbf{R(B_{|K|})\supseteq R(\outsource{B}{A}{K,L})}:$ Let $r\in R(\outsource{B}{A}{K,L})$.

If $r\in R_A\setminus \bigcup_{C\in K}RW(C)$ then because $r$ is a relation in the core (which is always $A$) and is never a generating relation for any clique that is to be outsourced, it will not be removed at any stage of the gradual outsourcing process.

If $r\in RW{\clqwit{C}{x^C,y^C}}$ for some $C\in K$, let $i$ be the stage when $C$ was outsourced. So $\clqwit{C}{x^C,y^C} \in T_j$ for all $j > i$ and this clique is never outsourced again. By definition of outsourcing the relations $RW{\clqwit{C}{x^C,y^C}}$ are present in any $B_{j}$ for $i>j$ and therefore $r\in R(B_{|K|})$.

If $r\in RW(C)$ for some $C\in L$ then because $L\subseteq T_j$ for all $j$ and the outsourcing is with respect to $T_j$, the relation $r$ is never removed from the structure and so $r\in R(B_{|K|})$.
\end{proof}

\bibliography{myrefs}{}
\bibliographystyle{plain}
\end{document}